\documentclass{amsart}
\usepackage{times}
\usepackage{dsfont,hyperref}
\usepackage[all]{xy}

\hypersetup{
    pdftoolbar=true,
    pdfmenubar=true,
    pdffitwindow=false,
    pdfstartview={FitH},
    pdftitle={On the equivariant implicit
function theorem with low regularity and applications to geometric variational problems},
    pdfauthor={Renato G. Bettiol, Paolo Piccione, and Gaetano Siciliano},
    pdfsubject={46T05, 47J07, 58C15, 58D19},
    pdfkeywords={}
    pdfnewwindow=true,
    colorlinks=true, %set this false for printable version
    linkcolor=blue,
    citecolor=blue,
    urlcolor=blue,
}

\newcommand{\im}{\mathrm{Im}\,}

\author{Renato G. Bettiol}
\author{Paolo Piccione}
\author{Gaetano Siciliano}

\address{
\begin{tabular}{lll}
University of Notre Dame & &Universidade de S\~ao Paulo \\
Department of Mathematics & & Departamento de Matem\'atica \\
255 Hurley Building & & Rua do Mat\~ao, 1010 \\
Notre Dame, IN, 46556-4618, USA & & S\~ao Paulo, SP, 05508-090, Brazil\\
\emph{E-mail address}: {\tt rbettiol@nd.edu} & & \emph{E-mail address}: {\tt piccione@ime.usp.br}\\
\end{tabular}
\bigskip
\hfill\break\hfill\indent
\begin{tabular}{lll}
Universidade Federal do ABC&&\\
Centro de Matem\'atica, Computa\c c\~ao e Cogni\c c\~ao &&\\
Rua Santa Ad\'elia 166&&\\
Santo Andr\'e, SP, 09210-170, Brazil&&\\
\emph{E-mail address}: {\tt gaetano.siciliano@gmail.com} & &
\end{tabular}
}

\title[On the $G$-equivariant implicit function theorem]{On the equivariant implicit
function theorem with low regularity and applications to geometric variational problems}
\date{November 12, 2012}
\thanks{The first named author is partially supported by the NSF grant DMS-0941615, USA. The second named author is partially supported by Fapesp, S\~ao Paulo, Brazil, Projeto Tem\'atico 2011/21362-2, and by CNPq, Brazil. The third named author is supported by Fapesp, S\~ao Paulo, Grant n.\ 2010/00068-6.}
\subjclass[2010]{46T05, 47J07, 58C15, 58D19}
\begin{document}

\theoremstyle{definition}\newtheorem*{defin*}{Definition}
\theoremstyle{plain}\newtheorem{teo}{Theorem}[section]
\theoremstyle{plain}\newtheorem{prop}[teo]{Proposition}
\theoremstyle{plain}\newtheorem{lem}[teo]{Lemma}
\theoremstyle{plain}\newtheorem{cor}[teo]{Corollary}
\theoremstyle{definition}\newtheorem{defin}[teo]{Definition}
\theoremstyle{remark}\newtheorem{rem}[teo]{Remark}
\theoremstyle{plain} \newtheorem{assum}[teo]{Assumption}
\theoremstyle{plain}\newtheorem{example}[teo]{Example}
\theoremstyle{plain} \newtheorem*{acknowledgement}{Acknowledgements}

\newtheorem{teon}{Theorem}
\renewcommand*{\theteon}{\Alph{teon}}

\newtheorem*{axiom}{Axiom}

\numberwithin{equation}{section}

\begin{abstract}
We prove an implicit function theorem for functions
on infinite-dimensional Banach manifolds, invariant under the (local) action of a finite
dimensional Lie group. Motivated by some geometric
variational problems,
we consider group actions that are not necessarily differentiable
everywhere, but only on some dense subset. Applications are discussed
in the context of harmonic maps, closed (pseudo-)Riemannian geodesics, and constant mean curvature
hypersurfaces.
\end{abstract}
\maketitle

\begin{section}{Introduction}
The implicit function theorem is an ubiquitous result from elementary multivariable
calculus courses to current pure and applied research problems. Being the literature on
the several formulations of the theorem virtually infinite, we will not attempt to
give an account of the extensive variety of statements available.
In this paper we formulate a version of the theorem for functions on Banach manifolds invariant under the action of a
finite-dimensional Lie group, which is not necessarily compact. Previous formulations
of the $G$-equivariant implicit function theorem, most notably by N. Dancer \cite{Dan0, Dan1, Dan2},
considered only \emph{linear} actions of groups on Banach spaces. This paper concerns two crucial improvements, namely:
\begin{itemize}
\item[(1)] the actions may be nonlinear;
\item[(2)] the action is only assumed to be by homeomorphisms, and possibly not everywhere differentiable.
\end{itemize}
Proving an implicit function theorem in such broad context is not just a matter
of abstract generality. Namely, as first noticed by R. Palais \cite{palais} and others in the 1960's,
the natural variational framework of several interesting
geometric problems involves functionals on Banach manifolds that are invariant under the
continuous (but not differentiable) action of a finite-dimensional Lie group of symmetries. Our result is
motivated precisely by this type of problem, that includes constant mean curvature (CMC)
embeddings, closed geodesics and harmonic maps, among other potential applications. A rough statement of our main abstract result is as follows.

\begin{teon}
Let $\mathfrak f\colon \mathfrak M\times\Lambda\to\mathds R$ be a map of class $\mathcal C^{k+1}$, $k\ge1$, where
$\mathfrak M$ and $\Lambda$ are (possibly infinite-dimensional) Banach manifolds, and assume that
for all $\lambda$, the functional $f(\cdot,\lambda)$ is invariant under the action of a finite-dimensional Lie group $G$
on $\mathfrak M$. Let $(x_0,\lambda_0)\in\mathfrak M\times\Lambda$ be such that
$\frac{\partial\mathfrak f}{\partial x}(x_0,\lambda_0)=0$. Assume the second variation $\frac{\partial^2\mathfrak f}{\partial x^2}(x_0,\lambda_0)$ is represented by a self-adjoint Fredholm operator and that the critical point $x_0$ is equivariantly nondegenerate, i.e.,
\[\ker\left(\frac{\partial^2\mathfrak f}{\partial x^2}(x_0,\lambda_0)\right)=T_{x_0}\big(G\cdot x_0\big).\]
Then, there exists a $\mathcal C^k$ map $x\colon U\subset\Lambda\to\mathfrak M$, defined in a neighborhood $U$
of $\lambda_0$ in $\Lambda$, with $x(\lambda_0)=x_0$, such that if $\lambda\in U$ and $y$ is sufficiently close
to the orbit $G\cdot x_0$, then $\frac{\partial\mathfrak f}{\partial x}(y,\lambda)=0$ if and only if $y$ belongs to the $G$-orbit
of $x(\lambda)$.
\end{teon}

The precise technical statement of the above result (see Theorem~\ref{thm:mainimplicitfunctiontheorem}) with all
detailed functional analytical conditions is cast in an abstract Banach vector bundle language, and given only later, in Section~\ref{sec:abstractGinvariantimplicit}, for the sake of exposition. Its formulation involves a set of axioms that describe a rather general setup to which the result applies. Axioms (A) describe the basic variational setup. Axiom (B) deals with the differentiability of the action,
and axiom (C) with the $G$-invariance of the set of critical points.
Axioms (D) give the existence of a \emph{gradient-like} map, and axioms (E) guarantee its equivariance.
Finally, axioms (F) provide a notion of continuity for the tangent space to the group
orbits.

Explicit applications to the above mentioned geometric variational problems
are discussed in Section~\ref{sec:applications}.
More precisely,
let us quote some of the deformation rigidity results that can be obtained as a direct consequence of our main theorem above:

\begin{teon}
Let $(\overline M,\overline{\mathbf g})$ be a smooth Riemannian manifold, let
$\mathbf g$ be a $\mathcal C^k$ Riemannian metric tensor on the manifold $M$, with $k\ge3$, and
let $\phi\colon M\to\overline M$ be a $(\mathbf g,\overline{\mathbf g})$-harmonic map, i.e., $\Delta_{\mathbf g,\overline{\mathbf g}}(\phi)=\mathrm{tr}(\widehat{\nabla}\mathrm d\phi)=0$.
Consider $\Lambda$ an open subset of a Banach space of
symmetric $(0,2)$-tensors of class $\mathcal C^k$ on $M$, with $\mathbf g\in\Lambda$,
such that every tensor in $\Lambda$ is a Riemannian metric tensor on $M$.
Suppose $\phi$ is nondegenerate, i.e., all Jacobi fields along $\phi$ are
of the form $\overline{K}\circ\phi$, where $\overline{K}$ is a Killing vector field of $\overline M$.
Then there exists a neighborhood $\mathcal U$ of $\mathbf g$ in
$\Lambda$, a neighborhood $\mathcal V$ of $\phi$ in $\mathcal C^{2,\alpha}(M,\overline M)$,
and a $\mathcal C^{k-1}$ function $\mathcal U\ni\mathbf h\mapsto\phi_{\mathbf h}\in\mathcal C^{2,\alpha}(M,\overline M)$
such that:
\begin{itemize}
\item[(a)] $\phi_{\mathbf h}$ is an $(\mathbf h,\overline{\mathbf g})$-harmonic map for all $\mathbf h\in\mathcal U$;
\item[(b)] if $\mathbf h\in\mathcal U$ and $\varphi\in\mathcal V$ is an
$(\mathbf h,\overline{\mathbf g})$-harmonic map, then $\varphi$ is geometrically equivalent to $\phi_\mathbf h$, i.e., there exists an isometry $\psi\in\mathrm{Iso}(\overline M,\overline{\mathbf g})$ such that $\varphi=\psi\circ\phi_\mathbf h$.
\end{itemize}
\end{teon}

\begin{teon}
Let $\mathbf g$ be a $\mathcal C^k$ (pseudo-)Riemannian metric tensor on the manifold $M$, with $k\ge3$, and
let $\gamma$ be a closed $\mathbf g$-geodesic on $M$.
Consider $\Lambda$ an open subset of a Banach space of
symmetric $(0,2)$-tensors of class $\mathcal C^k$ on $M$, with $\mathbf g\in\Lambda$,
such that every tensor in $\Lambda$ is a (pseudo-)Riemannian metric tensor on $M$.
Suppose $\gamma$ is nondegenerate, i.e., all periodic Jacobi fields along $\gamma$ are
(constant) multiples of the tangent field $\gamma'$. Then there exists a neighborhood
$\mathcal U$ of $\mathbf g$ in $\Lambda$, a neighborhood $\mathcal V$ of $\gamma$ in $\mathcal C^2(\mathds S^1,M)$,
and a $\mathcal C^{k-1}$ function $\mathcal U\ni\mathbf h\mapsto\gamma_{\mathbf h}\in\mathcal C^2(\mathds S^1,M)$
 such that:
\begin{itemize}
\item[(a)] $\gamma_{\mathbf h}$ is a closed $\mathbf h$-geodesic in $M$ for all $\mathbf h\in\mathcal U$;
\item[(b)] if $\mathbf h\in\mathcal U$ and $\alpha\in\mathcal V$ is a closed
$\mathbf h$-geodesic in $M$, then $\alpha$ is geometrically equivalent to $\gamma_\mathbf h$, i.e., $\alpha$ and $\gamma_\mathbf h$ have the same image (and the same number of turns).
\end{itemize}
\end{teon}

\begin{teon}
Let $x\colon M\hookrightarrow\overline M$ be a nondegenerate and transversely oriented
CMC embedding, with mean curvature $H_0$. Assume that there exists an invariant volume functional $\mathcal V$ defined in a neighborhood of $x$ in the set
of $\mathcal C^1$-embeddings of $M$ into $\overline M$. Then, there exists an open interval $\left]H_0-\varepsilon,H_0+\varepsilon\right[$ and a smooth
function $\left]H_0-\varepsilon,H_0+\varepsilon\right[\ni H\mapsto\varphi_H\in\mathcal C^{2,\alpha}(M)$, with
$\varphi_{H_0}=0$, such that:

\begin{itemize}
\item[(a)] for all $H\in\left]H_0-\varepsilon,H_0+\varepsilon\right[\,$, the map $x_H\colon M\hookrightarrow\overline M$
defined by \[x_H(p)=\exp_{x(p)}\big(\varphi_H(p)\cdot\vec n_x(p)\big),\qquad p\in M,\] is a CMC embedding having
mean curvature equal to $H$;
\smallskip

\item[(b)] any given CMC embedding $y\colon M\hookrightarrow\overline M$ sufficiently close to $x$
(in the $\mathcal C^{2,\alpha}$-topology) is isometrically congruent to some $x_H$.
\end{itemize}
\end{teon}

One of the main motivations for the development of our abstract result is the last theorem above, which is a generalization of some previous rigidity results for CMC embeddings. These previous results mostly originate from an idea of N. Kapouleas \cite{Kap1,Kap2},
which was then also employed by R. Mazzeo, F. Pacard and D. Pollack \cite{mpp}, R. Mazzeo and F. Pacard \cite{mp}, B. White \cite[\S 3]{Whi}, and, finally, J. P\'erez and A. Ros \cite[Thm~6.7]{PerRos96}. Using a similar idea, we prove a specific formulation of the $G$-equivariant implicit function theorem for CMC embeddings (see Proposition~\ref{thm:implfunctthmCMC}), without using Theorem A. The proof is purely geometric, based on a flux argument.
This curious proof cannot be extended, e.g., to the case where $x(M)$ is not the boundary of an open subset of $\overline M$.
In particular, this excludes the case of CMC embeddings of manifolds \emph{with boundary}.

Our abstract $G$-equivariant implicit function theorem applied to this setup covers a much broader situation, culminating in Theorem D, which generalizes Proposition~\ref{thm:implfunctthmCMC}. This is a practical illustration of the advantages of the generalized equivariant implicit function theorem given by Theorem A. Namely, the hypothesis that $x(M)$ is a boundary is replaced with the more general hypothesis that there exists a generalized volume functional in a $\mathcal C^1$-neighborhood of $x$ which is invariant under (small) isometries of the ambient space. Topological and geometrical conditions that guarantee that this hypothesis is satisfied are discussed in
Appendix~\ref{sec:invariantvolumes}. For instance, invariant volume functionals exist when
the ambient space is diffeomorphic to a sphere or to $\mathds R^n$; when the ambient space is not compact
but it has compact isometry group (in the case of embeddings of manifolds with boundary, compactness of
the isometry group is not necessary); or when the image of $x$ is contained in an open subset
of $\overline M$ which has vanishing de Rham cohomology in dimension $m=\mathrm{dim}(M)$.

Back to the abstract result, a final technical remark on its proof is in order.
The central point is the construction of a sort of \emph{slice}\footnote{%
The terminology here is not standard. Recall that a ``slice'' for an action
through a point is typically assumed invariant under the action of the isotropy of that point, see
\cite{Bredon}. This property is not required here.}
for the group action at a given smooth critical orbit of the variational problem.
More precisely, this is a smooth submanifold $S$, transversal to
the given smooth critical orbit, such that every nearby orbit (not necessarily smooth)
intercepts $S$, and with the property that it is a \emph{natural constraint}, i.e., restriction to $S$ of the variational
problem has the same critical points of the non-restricted functional.
Given the lack of regularity, transversality at an orbit does not imply
non-empty intersection with nearby orbits. The transversality argument is replaced
by a topological degree argument, that uses the finite-dimensionality of the group
orbits, see Proposition~\ref{thm:nonemptyintersection}.
\end{section}

\begin{section}{An implicit function theorem for CMC hypersurfaces}
\label{sec:implicitCMC}

In this section, we discuss an \emph{ad hoc} version of the implicit function theorem
in the context of constant mean curvature (CMC) embeddings in Riemannian manifolds, which
serves as motivation for the abstract formulation given in Section~\ref{sec:abstractGinvariantimplicit}.
 The basic setup is given by a CMC hypersurface $M$ of a Riemannian manifold $\overline M$.
Let us first recall two elementary applications of Stokes' theorem to the computation of
integrals involving Killing fields and mean curvature of submanifolds (see \cite[Lemma~5.5]{DanMir08}
for the $2$-dimensional orientable case).
\begin{lem}\label{thm:prelimzerointegral}
Let $(\overline M,\overline{\mathbf g})$ be a Riemannian manifold, let $M\subset\overline M$ be
a compact submanifold (without boundary), with mean curvature vector field $\vec H$,
and let $K\in\mathfrak X(M)$ be a Killing field in $M$.
Then:
\begin{equation}\label{eq::prelimzerointegral}
\int_M\overline{\mathbf g}(K,\vec H)\,\mathrm dM=0.
\end{equation}
In addition, if $M$ is the boundary of a (bounded) open subset of $\overline M$, then:
\begin{equation}\label{eq::prelimzerointegralbis}
\int_M\overline{\mathbf g}(K,\vec n)\,\mathrm dM=0,
\end{equation}
where $\vec n$ is a continuous unit normal field along $M$.
\end{lem}
\begin{proof}
Denote by $K_M\in\mathfrak X(M)$ the vector field on $M$ obtained by orthogonal projection
of $K$. We claim that $\mathrm{div}_M(K_M)=\overline{\mathbf g}(K,\vec H)$. Equality \eqref{eq::prelimzerointegral}
will then follow immediately from Stokes' Theorem.
In order to compute $\mathrm{div}_M(K_M)$, let $\overline\nabla$ denote the Levi-Civita connection of $\overline{\mathbf g}$
and let $\nabla$ be the Levi--Civita connection of the induced metric on $M$. If $\mathcal S$ is the
second fundamental form of $M$, then for all pairs $X,Y\in\mathfrak X(M)$, one has
$\overline\nabla_XY=\nabla_XY+\mathcal S(X,Y)$. Moreover, differentiating in the direction $X$ the equality $\overline{\mathbf g}(K_M,Y)=\overline{\mathbf g}(K,Y),$ we get:
\begin{equation}\label{eq:lemA}
\overline{\mathbf g}\big(\nabla_XK_M,Y\big)+\overline{\mathbf g}\big(K_M,\nabla_XY\big)
=\overline{\mathbf g}\big(\overline\nabla_XK,Y\big)+\overline{\mathbf g}\big(K,\overline\nabla_XY\big).
\end{equation}
Substituting $\overline{\mathbf g}(K,\overline\nabla_XY)=\overline{\mathbf g}(K,\nabla_XY)+\overline{\mathbf g}\big(K,\mathcal S(X,Y)\big)$ in \eqref{eq:lemA},
\begin{equation}\label{eq:lemB}
\overline{\mathbf g}\big(\nabla_XK_M,Y\big)=
\overline{\mathbf g}\big(\overline\nabla_XK,Y\big)+\overline{\mathbf g}\big(K,\mathcal S(X,Y)\big).
\end{equation}
Given $x\in M$, an orthonormal frame $e_1,\ldots,e_m$ of $T_xM$, and recalling that, since
$K$ is Killing, $\overline{\mathbf g}(\overline\nabla_{e_i}K,e_i)=0$
for all $i$, we get:
\[\mathrm{div}_M(K_M)=\sum_i\overline{\mathbf g}\big(\nabla_{e_i}K_M,e_i\big)=
\sum_i\overline{\mathbf g}\big(K,\mathcal S(e_i,e_i)\big)=\overline{\mathbf g}(K,\vec H),\]
which proves \eqref{eq::prelimzerointegral}. Formula \eqref{eq::prelimzerointegralbis} is an immediate application of Stokes' theorem, observing that $\mathrm{div}_{\overline M}K=0$, as $K$ is Killing.
\end{proof}

\begin{rem}\label{thm:topologicalcondition}
It is easy to find counterexamples to \eqref{eq::prelimzerointegralbis} when
$M$ is a hypersurface which is not the boundary of an open subset of $\overline M$.
If $M$ is the boundary of an open subset of $\overline M$, i.e., if the set
$\overline M\setminus M$ has two connected components, then $M$ is \emph{transversely oriented}.
This means that the normal bundle $TM^\perp$ is orientable. Conversely, if $M$ is transversely oriented,
then the condition that $M$ be the boundary of an open subset of $\overline M$  is equivalent to the condition
that the homomorphism $H_1(\overline M)\to H_1(\overline M,\overline M\setminus M)$ induced
in singular homology by the inclusion $(\overline M,\emptyset)\hookrightarrow(\overline M,\overline M\setminus M)$
be  trivial.
\end{rem}

\begin{defin}
Given a transversely oriented codimension one CMC embedding $x\colon M\hookrightarrow\overline M$,
the \emph{Jacobi operator} $J_x$ of $x$ is the second order linear elliptic differential
operator
\begin{equation}\label{eq:JacobiCMC}
J_x(f)=\Delta_xf-(m\,\mathrm{Ric}_{\overline M}(\vec{n}_x)+\|\mathcal S_x\|^2)f,
\end{equation}
defined on the space of $\mathcal C^2$-functions $f\colon M\to\mathds R$.
In the above formula, $m=\mathrm{dim}(M)$, $\Delta_x$ is the (positive) Laplacian of functions on $M$ relative to the pull-back metric $x^*(\overline{\mathbf g})$, $\mathrm{Ric}_{\overline M}(\vec{n}_x)$ is the Ricci curvature of $\overline M$ evaluated on
the unit normal field $\vec{n}_x$ of $x$ and $\mathcal S_x$ is the second fundamental form of $x$.
\end{defin}

\begin{defin}
A function $f$ satisfying $J_x(f)=0$ is called a \emph{Jacobi field} of $x$.
\end{defin}

\begin{rem}
It follows easily from \eqref{eq:JacobiCMC} that the space $\ker(J_x)$ of Jacobi fields of $x$ is a finite-dimensional space.
\end{rem}

\begin{rem}
Given any $\alpha\in\left]0,1\right[$, seen as a linear operator from $\mathcal C^{2,\alpha}(M)$
to $\mathcal C^{0,\alpha}(M)$, $J_x$ is a Fredholm map\footnote{
Second order self-adjoint elliptic operators acting on sections of Euclidean vector bundles over compact manifolds
are Fredholm maps of index zero from the space of $\mathcal C^{j,\alpha}$-sections to the space of $\mathcal C^{j-2,\alpha}$-sections,
$j\ge2$, see for instance \cite[\S 1.4]{Whi} and \cite[Theorem~1.1]{Whi2}. This fact will
be used throughout the paper.}
of index zero, which is symmetric
with respect to the $L^2$-pairing $\langle\cdot,\cdot\rangle_{L^2}\colon
\mathcal C^{2,\alpha}(M)\times\mathcal C^{0,\alpha}(M)\to\mathds R$, given by
$\langle f_1,f_2\rangle_{L^2}=\int_Mf_1\cdot f_2\,\mathrm dM.$
In particular, $\ker(J_x)=\im(J_x)^\perp$, relatively to the $L^2$-inner product.
\end{rem}

Note that if $\overline{K}$ is a Killing field of $(\overline M,\overline{\mathbf g})$, then
$f=\overline{\mathbf g}(\overline{K},\vec n_x)$ is a Jacobi field of $x$. The embedding $x$ is equivariantly nondegenerate if every Jacobi field arises in this way:

\begin{defin}
The CMC embedding $x$ is \emph{nondegenerate} if given any Jacobi field $f$ of $x$, there exists a Killing field
$\overline{K}$ of $(\overline M,\overline{\mathbf g})$ such that $f=\overline{\mathbf g}(\overline{K},\vec n_x)$.
\end{defin}

\begin{rem}
Nondegeneracy of every CMC embeddings of $M$ into $\overline M$ is a \emph{generic} property in the
set of Riemannian metrics $\overline{\mathbf g}$, see \cite{Whi, Whi2}.
\end{rem}

\begin{defin}
Two embeddings $x_i\colon M\hookrightarrow\overline M$, $i=1,2$, are said to be
\emph{congruent} if there exists a diffeomorphism $\phi\colon M\to M$ such that $x_2=x_1\circ\phi$, and
\emph{isometrically congruent} if there exists a diffeomorphism $\phi\colon M\to M$ and
an isometry $\psi\colon \overline M\to\overline M$ such that $x_2=\psi\circ x_1\circ\phi$.
Roughly speaking, congruence classes of embeddings of $M$ into $\overline M$ are
submanifolds of $\overline M$ that are diffeomorphic to $M$.
\end{defin}

We now prove the above mention formulation of the implicit function theorem for CMC embeddings (cf. Theorem D):

\begin{prop}\label{thm:implfunctthmCMC}
Let $x\colon M\hookrightarrow\overline M$ be a nondegenerate codimension one CMC embedding of a compact manifold
$M$ into a Riemannian manifold $(\overline M,\overline{\mathbf g})$, with mean curvature $H_0$. Assume also that
$x(M)$ is the boundary of an open subset of $\overline M$.
Then, there exists an open interval $\left]H_0-\varepsilon,H_0+\varepsilon\right[$ and a smooth
function $\left]H_0-\varepsilon,H_0+\varepsilon\right[\ni H\mapsto\varphi_H\in\mathcal C^{2,\alpha}(M)$, with
$\varphi_{H_0}=0$, such that:

\begin{itemize}
\item[(a)] for all $H\in\left]H_0-\varepsilon,H_0+\varepsilon\right[\,$, the map $x_H\colon M\hookrightarrow\overline M$
defined by \[x_H(p)=\exp_{x(p)}\big(\varphi_H(p)\cdot\vec n_x(p)\big),\qquad p\in M,\] is a CMC embedding having
mean curvature equal to $H$;
\smallskip

\item[(b)] any given CMC embedding $y\colon M\hookrightarrow\overline M$ sufficiently close to $x$
(in the $\mathcal C^{2,\alpha}$-topology) is isometrically congruent to some $x_H$.
\end{itemize}
\end{prop}
\begin{proof}
By a standard argument in submanifold theory, congruence classes of embeddings
$y\colon M\hookrightarrow\overline M$ near $x$ are parameterized by functions on $M$. More precisely,
to each function $\varphi\in\mathcal C^{2,\alpha}(M)$ one associates the map $x_\varphi\colon M\to\overline M$
defined by $x_\varphi(p)=\exp_{x(p)}\big(\varphi(p)\cdot\vec n_x(p)\big)$, $p\in M$. For $\varphi$ in a neighborhood
of $0$, $x_\varphi$ is an embedding of $M$ into $\overline M$. Conversely, given any embedding $y\colon M\hookrightarrow\overline M$
which is sufficiently close to $x$, there exists $\varphi\in\mathcal C^{2,\alpha}(M)$ near $0$ such that
$y$ is congruent to $x_\varphi$. Given a sufficiently small neighborhood $\mathcal U$ of $0$ in $\mathcal C^{2,\alpha}(M)$,
consider the map $\mathcal H\colon \mathcal U\to\mathcal C^{0,\alpha}(M)$ that associates
to each $\varphi$ the mean curvature function of the embedding $x_\varphi$. This function is smooth, as it is given
by a second order quasi-linear differential operator having smooth coefficients.
The derivative $\mathrm d\mathcal H(0)\colon \mathcal C^{2,\alpha}(M)\to\mathcal C^{0,\alpha}(M)$ coincides
with the Jacobi operator $J_x$.

By the nondegeneracy assumption on $x$, there exist $d=\dim\ker(J_x)\geq0$ Killing vector fields $\overline{K}_1,\ldots,\overline{K}_d$ of $(\overline M,\overline{\mathbf g})$ such that
the functions $f_i=\overline{\mathbf g}(\overline{K}_i,\vec n_x)$, $i=1,\ldots,d$, form a basis of $\ker(J_x)$.
Consider now the auxiliary map $\widetilde{\mathcal H}\colon \mathcal U\times\mathds R^d\to\mathcal C^{0,\alpha}(M)$
defined by:
\[\widetilde{\mathcal H}(f,a_1,\ldots,a_d)=\mathcal H(f)+\sum_{i=1}^da_if_i.\]
Clearly $\widetilde{\mathcal H}$ is smooth, and:
\[\mathrm d\widetilde{\mathcal H}(0)(f,b_1,\ldots,b_d)=J_x(f)+\sum_{i=1}^db_if_i.\]
Now, $\mathrm d\widetilde{\mathcal H}(0)$ is surjective; namely, the $f_i$'s span
the orthogonal complement of $\mathrm{Im}(J_x)$.
Moreover, the kernel of  $\mathrm d\widetilde{\mathcal H}(0)$ coincides with $\ker(J_x)\oplus\{0\}$,
which is finite-dimensional and therefore complemented in $\mathcal C^{2,\alpha}(M)\oplus\mathds R^d$.
In other words, $\widetilde{\mathcal H}$ is a smooth submersion at $0$.

Using the local form of submersions, we get that for $H$ near $H_0$, there
exists an open neighborhood $\mathcal V$ of $0$ in $\mathcal C^{2,\alpha}(M)\times\mathds R^d$ such that
the set $\widetilde{\mathcal H}^{-1}(H)\cap\mathcal V$ is a smooth embedded submanifold of dimension $d$.
Moreover, using the fact that submersions admit smooth local sections,
one has that there exists a smooth function $\left]H_0-\varepsilon,H_0+\varepsilon\right[\ni H\mapsto
\widetilde\varphi_H\in\mathcal V$ such that $\widetilde{\mathcal H}(\widetilde\varphi_H)=H$ for all $H$,
and with $\widetilde\varphi_{H_0}=0$.
Now, we claim that for all $H\in\mathds R$,
given $\widetilde\varphi=(\varphi,a_1,\ldots,a_d)\in\widetilde{\mathcal H}^{-1}(H)$,
then $a_1=\ldots=a_d=0$, and $\mathcal H(\varphi)=H$; in other words, $\widetilde{\mathcal H}^{-1}(H)=\mathcal H^{-1}(H)\times\{0\}$.
In order to prove the claim, assume
\[\mathcal H(\varphi)+\sum_{i=1}^da_if_i=H.\]
Multiplying both sides of this equality by $\sum_ia_if_i$ and integrating on $M$, keeping in mind
that:
\begin{equation*}
\int_M\mathcal H(\varphi)\sum_{i=1}^da_if_i\,\mathrm dM\stackrel{\eqref{eq::prelimzerointegral}}=0 \quad\mbox{ and }\quad H\cdot\int_M\sum_{i=1}^da_if_i\,\mathrm dM\stackrel{\eqref{eq::prelimzerointegralbis}}=0,
\end{equation*}
we get: $$\int_M\left[\sum_{i=1}^da_if_i\right]^2\,\mathrm dM=0.$$
This implies $a_1=\ldots=a_d=0$ and proves the claim.
Hence, we have $\widetilde\varphi_H=(\varphi_H,0,\ldots,0)$, with $H\mapsto\varphi_H$ satisfying item (a) of the proposition.

Item (b) also follows easily. Namely, the action by isometries of
$(\overline M,\overline{\mathbf g})$ on each CMC embedding $x_{\varphi_H}$ produces an orbit which is a
$d$-dimensional submanifold of the Banach space $\mathcal C^{2,\alpha}(M)$.\footnote{%
This is not a trivial fact, keeping into account that the left action of the
isometry group of $(\overline M,\mathbf g)$ on the space $\mathcal C^{2,\alpha}(M)$
obtained via exponential map of the normal bundle of $x$ is only continuous, and
not differentiable. However, it is proved in \cite{AliPic10} that the orbit
of any smooth embedding is a smooth submanifold.}
Such orbit is contained in $\mathcal H^{-1}(H)$, which is also a
$d$-dimensional submanifold around $x_H$. Hence,
a neighborhood of $x_H$ in the orbit of $x_H$
coincides with a neighborhood of $x_H$ in $\mathcal H^{-1}(H)$.
This implies that CMC embeddings $\mathcal C^{2,\alpha}$-close to $x$ must be
isometrically congruent to some $x_H$.
\end{proof}

\begin{rem}\label{thm:remnoboundary}
Observe that the assumption that $x(M)$ be the boundary of a bounded open subset of $\overline M$ cannot
be omitted in Proposition~\ref{thm:implfunctthmCMC}, as equality \eqref{eq::prelimzerointegralbis} is used
in the proof (see Remark~\ref{thm:topologicalcondition}). In particular, Proposition~\ref{thm:implfunctthmCMC}
does not cover the case of CMC embeddings of manifolds with boundary (cf. Theorem D).
\end{rem}
\end{section}

\begin{section}{\texorpdfstring{Statement of the $G$-equivariant implicit function theorem}{Statement of the G-equivariant implicit function theorem}}
\label{sec:abstractGinvariantimplicit}
The usual formulations of the implicit function theorem give a local result, so that
its statement can be given using open subsets of Banach spaces as domains and codomains
of the functions involved. For the equivariant version of the theorem that will be discussed
in this section the situation is somewhat different. Namely,  we will consider
group actions on Banach manifolds whose orbits are not necessarily contained in the
domain of some local chart, or in the domain of a local trivialization of a vector bundle.
In fact, we will not even assume boundedness of the orbits.
This suggests that, in spite of the local character of the result and its proof, the equivariant
formulation of our theorem is better cast in an abstract Banach manifolds/Banach vector bundles setup.

The basic setup is given by a manifold
$\mathfrak M$ acted upon by a Lie group $G$, another manifold $\Lambda$,
and a differentiable function $\mathfrak f\colon \mathfrak M\times\Lambda\to\mathds R$
which is $G$-invariant in the first variable.
More precisely, our framework is described by the following set of axioms:

\begin{axiom}[A1]
$\mathfrak M$ and $\Lambda$ are differentiable manifolds,
modeled on a (possibly infinite-dimensional) Banach space;
\end{axiom}

\begin{axiom}[A2]
 $G$ is a finite-dimensional Lie group, acting continuously on $\mathfrak M$ (on the left)
by homeomorphisms, and $\mathfrak g$ denotes its Lie algebra;
\end{axiom}

\begin{axiom}[A3]
$\mathfrak f\colon \mathfrak M\times\Lambda\to\mathds R$ is a function of class
$\mathcal C^{k+1}$, $k\ge1$, satisfying $\mathfrak f(g\cdot x,\lambda)=\mathfrak f(x,\lambda)$ for all
$g\in G$, $x\in\mathfrak M$ and $\lambda\in\Lambda$.
\end{axiom}

For all $x\in\mathfrak M$, denote by
\begin{equation}\label{eq:defbetax}
\beta_x\colon G\longrightarrow\mathfrak M, \quad \mbox{ and }\quad \gamma_g\colon \mathfrak M\longrightarrow\mathfrak M,
\end{equation}
the map $\beta_x(g)=g\cdot x$ and the homeomorphism $\gamma_g(x)=g\cdot x$, respectively.

As to the regularity of the group action, we make the following assumptions:

\begin{axiom}[B]
There exists a dense subset $\mathfrak M'\subset\mathfrak M$
such that for all $x\in\mathfrak M'$ the map $\beta_x\colon G\to\mathfrak M$ is differentiable
at $1\in G$.
\end{axiom}

Let us denote by $\partial_1\mathfrak f\colon \mathfrak M\times\Lambda\to
T\mathfrak M^*$ the derivative of $\mathfrak f$ with respect to the first variable; our aim is to study
the equation $\partial_1\mathfrak f(x,\lambda)=0$.
Observe that with the weak regularity assumptions on the group action
(we do not assume in principle the differentiability of the map $\gamma_g$), it does not follow
that if $\partial_1\mathfrak f(x,\lambda)=0$, then also $\partial_1\mathfrak f(g\cdot x,\lambda)=0$ for all $g\in G$.
We will therefore explicitly assume:

\begin{axiom}[C]
For all $\lambda\in\Lambda$, the set $\big\{x\in\mathfrak M:\partial_1\mathfrak f(x,\lambda)=0\big\}$
is $G$-invariant.
\end{axiom}

Let us now look at the question of lack of a \emph{gradient} for the function
$\mathfrak f$; we define a \emph{gradient-like} map by introducing a suitable vector bundle
on the manifold $\mathfrak M$, defined by the following axioms:

\begin{axiom}[D1]
$\mathcal E\to\mathfrak M$ is a $\mathcal C^k$-Banach vector bundle;
\end{axiom}

\begin{axiom}[D2]
There exist $\mathcal C^k$-vector bundle morphisms:
\[\mathfrak i\colon T\mathfrak M\longrightarrow\mathcal E\quad\text{and}\quad\mathfrak j\colon \mathcal E\longrightarrow T\mathfrak M^*,\]
with $\mathfrak j$ injective;
\end{axiom}

\begin{axiom}[D3]
For all $x\in\mathfrak M$, the bilinear form $\langle\cdot,\cdot\rangle_x\colon T_x\mathfrak M\times T_x\mathfrak M\to\mathds R$
defined by $\langle u,v\rangle_x=\mathfrak j_x\big(\mathfrak i_x(u)\big)v$ is a (not necessarily complete) positive definite
inner product (this implies that also $\mathfrak i$ is injective);
\end{axiom}

\begin{axiom}[D4]
There exists a $\mathcal C^k$-map $\delta\mathfrak f\colon \mathfrak M\times\Lambda\to\mathcal E$ such that
\[\mathfrak j\circ\delta\mathfrak f=\partial_1\mathfrak f.\]
\end{axiom}

Since $\mathfrak j$ is injective, we get that $\partial_1\mathfrak f(x,\lambda)=0$
if and only if $\delta\mathfrak f(x,\lambda)=0$.

Let us now go back to the question of $G$-invariance of the set of
critical points of the functions $\mathfrak f(\cdot,\lambda)$. Assuming that the $G$ action
is by diffeomorphisms (i.e., that the maps $\gamma_g$ are diffeomorphisms), given
$x_0$ such that $\partial_1\mathfrak f(x_0,\lambda)=0$ then obviously $\partial_1\mathfrak f(g\cdot x_0,\lambda)=0$
for all $g\in G$. For this conclusion it is necessary to differentiate $\gamma_g$;
when the action of $G$ is only by homeomorphisms, the $G$-invariance of the critical set
is obtained under a suitable assumption of $G$-equivariance for the map $\delta\mathfrak f$.
Given $x\in\mathfrak M$, the fiber of $\mathcal E$ over $x$ will be denoted by $\mathcal E_x$.

\begin{axiom}[E1]
There exists a continuous left $G$-action by linear isomorphisms on the fibers of
$\mathcal E$ compatible with the action on $\mathfrak M$, i.e.,
such that the projection $\mathcal E\to\mathfrak M$ is equivariant (this means that for each $g$, it is given a family
of linear isomorphisms $\varphi_{g,x}\colon \mathcal E_x\to\mathcal E_{g\cdot x}$ depending continuously on $x\in\mathfrak M$ and on
$g\in G$, such that $\varphi_{gh,x}=\varphi_{g,h\cdot x}\circ\varphi_{h,x}$ for all $g,h\in G$ and all $x\in\mathfrak M$);
\end{axiom}

\begin{axiom}[E2]
The map $\delta\mathfrak f(\cdot,\lambda)\colon \mathfrak M\to\mathcal E$ is equivariant for all $\lambda\in\Lambda$.
\end{axiom}

\begin{lem}\label{thm:criticalorbit}
Axioms {\rm (E1)} and {\rm (E2)} imply {\rm (C)}.
\end{lem}
\begin{proof}
Assume $\partial_1\mathfrak f(x_0,\lambda)=0$, then $\delta\mathfrak f(x_0,\lambda)=0$.
The equivariance property gives $\delta\mathfrak f(g\cdot x_0,\lambda)=0$ for all $g\in G$, i.e.,
$\partial_1\mathfrak f(g\cdot x_0,\lambda)=0$ for all $g\in G$.
\end{proof}
Finally, another set of assumptions is needed in order to deal with the lack
of the map $x\mapsto\mathrm d\beta_x(1)\in\mathrm{Lin}(\mathfrak g,T\mathfrak M)$ for all $x\in\mathfrak M$.
Our next set of hypotheses will give the existence of a continuous extension to $\mathfrak M$ of
this map provided that its codomain be enlarged and endowed with a weaker topology.
As above, this set of assumptions is better cast in terms of vector bundles and injective morphisms:

\begin{axiom}[F]
There exists a $\mathcal C^k$-vector bundle $\mathcal Y\to\mathfrak M$ and $\mathcal C^k$-vector
bundle morphisms:
\[\widetilde{\jmath}\colon \mathcal E\longrightarrow\mathcal Y^*\quad\text{and}\quad\kappa\colon T\mathfrak M\longrightarrow\mathcal Y,\]
with $\kappa$ injective, such that:
\begin{itemize}
\item[(F1)] $\kappa^*\circ\widetilde{\jmath}=\mathfrak j$ (from which it follows that also $\widetilde{\jmath}$ is injective);
\item[(F2)] the map $\mathfrak M'\ni x\mapsto\kappa_x\circ\mathrm d\beta_x(1)\in\mathrm{Lin}(\mathfrak g,\mathcal Y_x)$
has a continuous extension to a section of the vector bundle $\mathrm{Lin}(\mathfrak g,\mathcal Y)\to\mathfrak M$.
\end{itemize}
\end{axiom}

From density of $\mathfrak M'$, the extension in (F2) is therefore unique.
We are now ready for the detailed technical statement of Theorem A in the Introduction and its proof:

\begin{teo}
\label{thm:mainimplicitfunctiontheorem}
In the above setup of Axioms (A1) to (F), let  $(x_0,\lambda_0)\in\mathfrak M'\times\Lambda$ be a point such that $\partial_1\mathfrak f(x_0,\lambda_0)=0$. Denote by $L\colon T_{x_0}\mathfrak M\to\mathcal E_{x_0}$ the linear map
\[L:=\pi_{\mathrm{ver}}\circ \partial_1(\delta\mathfrak f)(x_0,\lambda_0),\]
where $\pi_{\mathrm{ver}}\colon T_{0_{x_0}}\mathcal E\to\mathcal E_{x_0}$ is the canonical vertical projection. If:
\begin{itemize}
\item[(G1)] $L$ is Fredholm of index $0$;
\item[(G2)] $\ker L=\im\mathrm d\beta_{x_0}(1)$,
\end{itemize}
then, there exists a $G$-invariant neighborhood $V\subset\mathfrak M\times\Lambda$
of $(G\cdot x_0,\lambda_0)$ and a $\mathcal C^k$-function $\sigma\colon \Lambda_0\to\mathfrak M$ defined
in a neighborhood $\Lambda_0$ of $\lambda_0$ in $\Lambda$ such that
$(x,\lambda)\in V$ and $\partial_1\mathfrak f(x,\lambda)=0$ hold if and only if $x\in G\cdot\sigma(\lambda)$.
\end{teo}

\begin{rem}
Condition (G2) is an equivariant {\em nondegeneracy condition} on the critical orbit $G\cdot x_0$.
\end{rem}

\begin{proof}
We will study a local problem first, and we will then use the group action for the proof
of the global statement. After suitable local charts and local trivialization of vector bundles
around the point $(x_0,\lambda_0)$ have been chosen, one can assume the following situation:
\begin{itemize}
\item $\mathfrak M$ is an open subset of a Banach space $X$, $\mathfrak M'$ is a dense subset
of $\mathfrak M$ which is endowed with a topology finer than the induced topology from $\mathfrak M$, and
$\Lambda$ is an open subset of another Banach space;
\item the group action on $\mathfrak M$ is described by a map $\mathcal U\ni(g,x)\mapsto g\cdot x\in\mathfrak M$,
with $\mathcal U$ an open neighborhood of $\{1\}\times\mathfrak M$ in $G\times\mathfrak M$. Such map
satisfies the obvious equalities given by group operations whenever\footnote{
For instance, the equality $g\cdot(h\cdot x)=(gh)\cdot x$ holds for all $g,h\in G$ and $x\in\mathfrak M$
such that $(h,x)\in\mathcal U$ and $(g,h\cdot x)\in\mathcal U$. In particular, given $x\in\mathfrak M$,
the equality must hold when $g$ and $h$ belong to some neighborhood of $1$ in $G$.
This could be formalized in terms of \emph{partial actions} of groups (or groupoids) on topological spaces, but this is not
relevant in the context of the present paper.} such equalities make sense in the open set $\mathcal U$;
\item the $\mathcal C^{k+1}$-function $\mathfrak f\colon \mathfrak M\times\Lambda\to\mathds R$ satisfies
$\mathfrak f(g\cdot x,\lambda)=\mathfrak f(x,\lambda)$ wherever such equality makes sense (as above);
\item the vector bundle $\mathcal E$ is replaced with the product $\mathfrak M\times\mathcal E_0$, where
$\mathcal E_0$ is a fixed Banach space (isometric to the typical fiber of $\mathcal E$);
\item $\mathfrak j\colon \mathfrak M\to\mathrm{Lin}(\mathcal E_0,X^*)$ is a $\mathcal C^k$-map
such that $\mathfrak j_x$ is injective for all $x\in\mathfrak M$;
\item $\mathfrak i\colon \mathfrak M\to\mathrm{Lin}(X,\mathcal E_0)$ is a $\mathcal C^k$-map such that
$\mathfrak j_x\circ\mathfrak i_x\colon X\to X^*$ is a (not necessarily complete) positive definite inner product
on $X$ (which implies in particular that $\mathfrak i_x$ is injective for all $x$);
\item the vector bundle $\mathcal Y$ is replaced with the product $\mathfrak M\times\mathcal Y_0$, where
$\mathcal Y_0$ is a fixed Banach space (isometric to the typical fiber of $\mathcal Y$);
\item $\widetilde\jmath\colon \mathfrak M\to\mathrm{Lin}(\mathcal E_0,\mathcal Y_0^*)$
and $\kappa\colon \mathcal M\to\mathrm{Lin}(X,\mathcal Y_0)$ are $\mathcal C^k$-maps taking values
in the set of injective linear maps, and such that $\kappa_x^*\circ\widetilde{\jmath}_x=\mathfrak j_x$
for all $x\in\mathfrak M$;
\item for all $x\in\mathfrak M$, the map $\beta_x$ is only defined on an open neighborhood of $1$ in $G$.
For $x\in\mathfrak M'$, its derivative at $1$ is a linear map $\mathrm d\beta_x(1)\colon \mathfrak g\to X$
that depends continuously on $x$, relatively to the finer topology of $\mathfrak M'$;
\item the map $\mathfrak M'\ni x\mapsto\kappa_x\circ\big[\mathrm d\beta_x(1)\big]\in\mathrm{Lin}(\mathfrak g,\mathcal Y_0)$
has a continuous extension to $\mathfrak M$;
\item $\partial_1\mathfrak f\colon \mathfrak M\times\Lambda\to X^*$ and $\delta\mathfrak f\colon \mathfrak M\times\Lambda\to\mathcal E_0$
are maps of class $\mathcal C^k$ such that
$\mathfrak j_x\big(\delta\mathfrak f(x,\lambda)\big)=\partial_1\mathfrak f(x,\lambda)$ for all $(x,\lambda)$;
\item the linear operator $L\colon X\to\mathcal E_0$ is given by the partial derivative $\partial_1(\delta\mathfrak f)(x_0,\lambda_0)$. It is a Fredholm operator of index zero, and $\ker L$ is given
by the image of the linear map $\mathrm d\beta_{x_0}(1)$.
\end{itemize}
Let $S=\mathrm{Im}\big(\mathrm d\beta_{x_0}(1)\big)^\perp$ be the closed subspace of $X$ given by the orthogonal complement of the subspace
$\mathrm{Im}\big(\mathrm d\beta_{x_0}(1)\big)$ relatively to the inner product $\langle\cdot,\cdot\rangle=\mathfrak j_{x_0}\circ\mathfrak i_{x_0}$.
Since $\mathrm{Im}\big(\mathrm d\beta_{x_0}(1)\big)$ is finite-dimensional,
we have a direct sum decomposition
\begin{equation}\label{eq:splittingX}
X=\mathrm{Im}\big(\mathrm d\beta_{x_0}(1)\big)\oplus S.
\end{equation}
Let us now introduce a finite-dimensional subspace $Y\subset\mathcal E_0$ by:
\[Y=\mathfrak i_{x_0}\big(\mathrm{ker}\, L\big);\]
we claim that $Y$ is complementary to the closed subspace $\im L$ in $\mathcal E_0$. In order to prove the claim,
we first observe that, using the fact that $\mathfrak i_{x_0}$ is injective and $L$ has index $0$, then the dimension of $Y$ equals the codimension of $\im L$. Thus, our claim is proved
if we show that $Y\cap\im L=\{0\}$. We have a commutative diagram:
\begin{equation}\label{eq:firstcommdiagr}
\xymatrix{X\ar[rr]^{L}\ar[rrd]_{\partial_1(\partial_1\mathfrak f)(x_0,\lambda_0)\quad}&&\mathcal E_0\ar[d]^{\mathfrak j_{x_0}}
\cr & &X^*}
\end{equation}
that is easily obtained differentiating the equality $\mathfrak j_x\big(\delta\mathfrak f(x,\lambda_0)\big)=
\partial_1\mathfrak f(x,\lambda_0)$ with respect to $x$ at $x=x_0$, keeping in mind that
$\delta\mathfrak f(x_0,\lambda_0)=0$. Observe that the second line in \eqref{eq:firstcommdiagr} is a symmetric operator,
and therefore we obtain:
\begin{equation}\label{eq:inclusionimageannihilator}
\mathfrak j_{x_0}\big(\im L\big)\subset\big[\mathrm{ker}\,(\mathfrak j_{x_0}\circ L)\big]^o=
\big(\mathrm{ker}\, L\big)^o,
\end{equation}
where $W^o$ denotes the annihilator of the subspace $W\subset X$ in $X^*$.
Now, if $v\in\ker L$ is such that $\mathfrak i_{x_0}(v)\in\im L$, then by
\eqref{eq:inclusionimageannihilator}, $\mathfrak j_{x_0}\circ\mathfrak i_{x_0}(v)\in \big(\mathrm{ker}\, L\big)^o$,
and in particular $\mathfrak j_{x_0}\big(\mathfrak i_{x_0}(v)\big)v=0$. By assumption (D3),
it follows that $v=0$, i.e., $Y\cap\im L=\{0\}$, and therefore:
\[\mathcal E_0=Y\oplus\im L.\]
Let $P\colon \mathcal E_0\to\im L$ be the projection relative to this direct sum decomposition
of $\mathcal E_0$. We define the function:
\begin{equation*}
\begin{aligned}
H\colon (\mathfrak M\cap S)\times\Lambda\longrightarrow\im L \\
H(x,\lambda)=P\big(\delta\mathfrak f(x,\lambda)\big)
\end{aligned}
\end{equation*}
observe that $H(x_0,\lambda_0)=0$.
Such map has the same regularity as $\delta\mathfrak f$. The derivative $\partial_1H(x_0,\lambda_0)$
is $P\circ L\vert_S=L\vert_S\colon S\to\im L$, and this is an isomorphism by assumption (G2) and
by \eqref{eq:splittingX}. We can therefore apply the standard implicit function theorem to the
equation $H(x,\lambda)=0$ around $(x_0,\lambda_0)$, obtaining a neighborhood $\Lambda_0$ of
$\lambda_0$ in $\Lambda$ and a $\mathcal C^k$-function $\sigma\colon \Lambda_0\to(\mathfrak M\cap S)$
with $\sigma(\lambda_0)=x_0$ and such that, given $(x,\lambda)$ in a neighborhood of $(x_0,\lambda_0)$
in $(\mathfrak M\cap S)\times\Lambda$, the equality $H(x,\lambda)=0$ holds if and only if $x=\sigma(\lambda)$.

In order to complete the proof of our theorem, we will show the following facts:
\begin{enumerate}
\item\label{itm:fact1} there exists a neighborhood $W$ of $(x_0,\lambda_0)$ in $\mathfrak M\times\Lambda$ such that,
given $(x,\lambda)\in W$, then $H(x,\lambda)=0$ if and only if $\partial_1\mathfrak f(x,\lambda)=0$;
\item\label{itm:fact2} if $x\in\mathfrak M$ is sufficiently close to $x_0$, then the orbit $G\cdot x$ has non-empty intersection
with $\mathfrak M\cap S$.
\end{enumerate}
By possibly reducing the domain of the function $\sigma$, we can assume that its graph is contained
in $W$. The first claim implies that given $(x,\lambda)$ sufficiently close to $(x_0,\lambda_0)$
in $(\mathfrak M\cap S)\times\Lambda$, the equality $\partial_1\mathfrak f(x,\lambda)=0$ holds if and
only if $x=\sigma(\lambda)$.
The second claim and assumption (C) will imply that given $(x,\lambda)$ sufficiently close to $G\cdot x_0\times\{\lambda_0\}$
in $\mathfrak M\times\Lambda$, then $\partial_1\mathfrak f(x,\lambda)=0$ if and only if $x\in G\cdot\sigma(\lambda)$.

In order to prove claim \eqref{itm:fact1}, we observe first that if $(x,\lambda)\in(\mathfrak M\cap S)\times\Lambda$ and
$\partial_1\mathfrak f(x,\lambda)=0$,
then $\delta\mathfrak f(x,\lambda)=0$ and therefore $H(x,\lambda)=0$. Conversely, let us show that
if $x\in\mathfrak M\cap S$ is near $x_0$, and $H(x,\lambda)=0$, then $\delta\mathfrak f(x,\lambda)=0$
(and thus also $\partial_1\mathfrak f(x,\lambda)=0$).
We observe that if $H(x,\lambda)=0$, then $\delta\mathfrak f(x,\lambda)\in
\mathfrak i_{x_0}\big(\mathrm{ker}\, L\big)$,
and thus, $\widetilde\jmath_x\big(\delta\mathfrak f(x,\lambda)\big)$ \emph{annihilates}
$\big[\widetilde\jmath_{x}\big(\mathfrak i_{x_0}(\ker L)\big)\big]_o$.
Here, given a subspace $Z\subset X^*$, the symbol $Z_o$ denotes the subspace of $X$
annihilated by $Z$. Denote by $B\colon \mathfrak M\to\mathrm{Lin}(\mathfrak g,\mathcal Y_0)$ the
continuous extension of the map $x\mapsto\kappa_x\circ\big[\mathrm d\beta_x(1)\big]$
defined in $\mathfrak M'$ (by (F2)); we claim that $\widetilde\jmath_x\big(\delta\mathfrak f(x,\lambda)\big)$
annihilates also the image of $B(x)$. This follows from the fact that
for $x\in\mathfrak M'$, $\partial_1\mathfrak f(x,\lambda)$ annihilates $\mathrm{Im}\big(\mathrm d\beta_x(1)\big)$, which
is easily seen by differentiating at $g=1$ the (constant) map $g\mapsto\mathfrak f\big(\beta_x(g),\lambda\big)$
(use (B1)), and a continuity argument.
Namely, observe that for $x\in\mathfrak M'$:
\begin{multline*}0=\partial_1\mathfrak f(x,\lambda)\circ\mathrm d\beta_x(1)=\mathfrak j_x\big(\delta\mathfrak f(x,\lambda)\big)
\circ[\mathrm d\beta_x(1)]\\=\kappa_x^*\big(\widetilde\jmath_x\big(\delta\mathfrak f(x,\lambda)\big)\big)\circ
[\mathrm d\beta_x(1)]=
\widetilde\jmath_x\big(\delta\mathfrak f(x,\lambda)\big)\circ\kappa_x\circ[\mathrm d\beta_x(1)].\end{multline*}
This says that the map:
\[\mathfrak M\ni x\longmapsto\widetilde\jmath_x\big(\delta\mathfrak f(x,\lambda)\big)\circ B(x)\in\mathfrak g^*\]
vanishes for $x\in\mathfrak M'$. Thus, by continuity, it vanishes identically, i.e.,
$\widetilde\jmath_x\big(\delta\mathfrak f(x,\lambda)\big)$ annihilates the image of $B(x)$.

To conclude the proof of \eqref{itm:fact1}, it suffices
to show that for $x\in\mathfrak M$ near $x_0$ one has:
\begin{equation}\label{eq:opencondition}
\mathrm{Im}(B(x))+\big[\widetilde{\jmath}_{x}\big(\mathfrak i_{x_0}(\ker L)\big)\big]_o=\mathcal Y_0.
\end{equation}
Using the continuity of $B$ and the fact that the subspace $\big[\widetilde{\jmath}_{x}\big(\mathfrak i_{x_0}(\ker L)\big)\big]_o$
has \emph{fixed} codimension in $\mathcal Y_0$ (i.e., it does not depend
on $x$), it follows that condition \eqref{eq:opencondition} is open\footnote{%
Details on the proof of openness of condition \eqref{eq:opencondition} are as follows.
Let $e_1,\ldots,e_r$ be a basis of $\ker L$; the covectors
$\omega_i=\widetilde\jmath_x\big(\mathfrak i_{x_0}(e_i)\big)$, $i=1,\ldots,r$,
are linearly independent in $\mathcal Y_0^*$. Consider the surjective
linear map $\tau_x\colon \mathcal Y_0\to\mathds R^r$
defined by $\tau_x(v)=\big(\omega_1(v),\ldots,\omega_r(v)\big)$.
The map $\mathfrak M\ni x\mapsto\tau_x\in\mathrm{Lin}(\mathcal Y_0,\mathds R^r)$ is continuous.
Condition \eqref{eq:opencondition} is equivalent to $\mathrm{Im}\big(B(x)\big)+\ker\tau_x=\mathcal Y_0$,
i.e., that the linear map $\tau_x\circ B(x):\mathfrak g\to\mathds R^r$ be surjective. This is clearly an open
condition.}
in $\mathfrak M$. Thus, it suffices
to show that it holds at $x=x_0$. Let us check equality \eqref{eq:opencondition} at $x=x_0$.
The dimension of $\mathrm{Im}(B(x_0))$ equals the dimension of $\mathrm{Im}\big(\mathrm d\beta_{x_0}(1)\big)$,
and this is equal to the dimension of $\ker L$, by assumption
(G2). Since $\widetilde\jmath_{x_0}$ and $\mathfrak i_{x_0}$ are injective, then
the codimension of $\big[\widetilde{\jmath}_{x}\big(\mathfrak i_{x_0}(\ker L)\big)\big]_o$ is equal
to the dimension of $\ker L$. Thus, it suffices to show:
\[\kappa_{x_0}\big(\mathrm{Im}\big(\mathrm d\beta_{x_0}(1)\big)\big)\cap\big[\widetilde{\jmath}_{x_0}\big(\mathfrak i_{x_0}(\ker L)\big)\big]_o=\{0\},\]
i.e.,
\begin{multline*}\mathrm{Im}\big(\mathrm d\beta_{x_0}(1)\big)\cap\kappa_{x_0}^{-1}
\big[\widetilde{\jmath}_{x_0}\big(\mathfrak i_{x_0}(\ker L)\big)\big]_o=
\ker L\cap\big[\kappa_{x_0}^*\circ\widetilde{\jmath}_{x_0}\big(\mathfrak i_{x_0}(\ker L\big)\big]_o\\
=
\ker L\cap\big[\mathfrak j_{x_0}\big(\mathfrak i_{x_0}(\ker L)\big)\big]_0=
\ker L\cap(\ker L)^\perp=
\{0\}.
\end{multline*}

It remains to show claim \eqref{itm:fact2}, i.e., equivalently,
that the set $G\cdot S$ contains an open neighborhood of $x_0$. Since we are not assuming differentiability
of the group action, this does not follow from a transversality argument. The correct argument in the
continuous case uses the notion of topological degree of a map, and it will be given separately in
Proposition~\ref{thm:nonemptyintersection}, below. In our case, this result is used setting $A=\mathfrak M$, $M=G$,
$N=\mathfrak M$, $P=S$, $m_0=1$, $a_0=x_0$ and $\chi$ is the action.
\end{proof}

\begin{prop}\label{thm:nonemptyintersection}
Let $N$ be a (possibly infinite-dimensional) Banach manifold, $P\subset N$ a Banach submanifold, $M$ a finite-dimensional manifold and $A$ a topological space. Assume $\chi\colon A\times M\to N$ is a continuous
function such that there exists $a_0\in A$ and $m_0\in M$ with:
\begin{itemize}
\item[(a)] $\chi(a_0,m_0)\in P$;
\item[(b)] $\chi(a_0,\cdot)\colon M\to N$ is of class $\mathcal C^1$;
\item[(c)] $\partial_2\chi(a_0,m_0)\big(T_{m_0}M\big)+T_{\chi(a_0,m_0)}P=T_{\chi(a_0,m_0)}N$.
\end{itemize}
Then, for $a\in A$ near $a_0$, $\chi(a,M)\cap P\ne\emptyset$.
\end{prop}
\begin{proof}
Given a function $f\colon U\subset\mathds R^d\to\mathds R^d$ of class $\mathcal C^1$, where $U$ is an open
neighborhood of $0$, such that $f(0)=0$ and $\mathrm df(0)$ an isomorphism, then the induced
map $\widetilde f\colon \mathds S^{d-1}\to\mathds S^{d-1}$ has topological degree equal to $\pm1$.
Here, $\widetilde f$ is defined by $\widetilde f(x)=\Vert f(xr)\Vert^{-1}f(xr)$, where $r>0$ is
such that $0$ is the unique zero of $f$ in the closed ball $B[0;r]$ of $\mathds R^d$.

Now, if $A$ is any topological space, $f\colon A\times U\to\mathds R^d$ is continuous, and $a_0\in A$ is
such that $f(a_0,\cdot)$ is of class $\mathcal C^1$, $f(a_0,0)=0$ and $\partial_2f(a_0,0)$ is an isomorphism,
for $a$ near $a_0$, and $r>0$ sufficiently small, $0\in f\big(a,B[0;r]\big)$. This follows from the
continuity of the topological degree. The same holds for a function $f\colon A\times U\to\mathds R^d$,
where now $U$ is an open neighborhood of $0$ in $\mathds R^s$, with $s\ge d$, under the assumption
that $f(a_0,\cdot)$ be of class $\mathcal C^1$, $f(a_0,0)=0$, and $\partial_2f(a_0,0)$ be surjective.
Namely, it suffices to apply the argument above to the function obtained by restricting $f$ to a $d$-dimensional
subspace where $\partial_2f(a_0,0)$ is an isomorphism.

To prove the result, use local coordinates adapted to $P$ in $N$, and assume that $M$, $P$ and $N$
are Banach spaces, with $N=P\oplus\mathds R^d$,  $d\le s=\mathrm{dim}(M)$ is the codimension of $P$, and $m_0=0$.
In this situation, the result is obtained applying the argument above to the function
$f\colon A\times M\to\mathds R^d$ given by $f(a,m)=\pi\big(\chi(a,m)\big)$, where $\pi\colon N\to\mathds R^d$ is the
projection relative to the decomposition $N=P\oplus\mathds R^d$. Clearly, $f(a,m)=0$ if and only if
$\chi(a,m)\in P$. Assumption (a) implies
that $f(a_0,0)=0$, and assumption (c) implies that $\partial_2f(a_0,0)$ is surjective.
\end{proof}
\begin{rem}\label{thm:remlocalpartial}
In Theorem~\ref{thm:mainimplicitfunctiontheorem}, some assumptions on the group action can be weakened.
For instance, the result holds also for \emph{local} group actions, see Appendix~\ref{sec:localactions}.
This version of the equivariant implicit function theorem will be used in the constant mean curvature problem,
Subsection~\ref{subsec:cmc}.
Versions of the result for the so-called \emph{partial actions} of groups, or even for actions of groupoids,
semigroups, monoids, etc., are also possible.
\end{rem}
\end{section}

\begin{section}{Applications to geometric variational problems}\label{sec:applications}
We now describe concrete applications of our abstract result to three classic geometric variational problems:
 harmonic maps,
closed geodesics, and
constant mean curvature hypersurfaces,
corresponding to Theorems B, C and D in the Introduction, respectively.

\subsection{Harmonic maps}
Let $(M,\mathbf g)$ and $(\overline M,\overline{\mathbf g})$ be Riemannian manifolds.

\begin{defin}
A $\mathcal C^2$-map $\phi\colon M\to\overline M$
is said to be $(\mathbf g,\overline{\mathbf g})$-\emph{harmonic} if
\begin{equation}\label{eq:deflaplacian}
\Delta_{\mathbf g,\overline{\mathbf g}}(\phi):=
\mathrm{tr}(\widehat{\nabla}\mathrm d\phi)=0
\end{equation}
where $\widehat{\nabla}$ is the
connection on the vector bundle $TM^*\otimes\phi^*(T\overline{M})$ over $M$
induced by the Levi-Civita connections $\nabla$ of $\mathbf g$ and $\overline{\nabla}$ of $\overline{\mathbf g}$.
\end{defin}

\begin{rem}
Harmonic maps form a class that contains several geometrically important objects, see \cite{eellslemaire}.
For instance, if $\dim M=1$, harmonic maps $\phi\colon M\to\overline M$ are the geodesics of $\overline M$.
In particular, setting $M=\mathds S^1$, the previous statements in Subsection~\ref{subsec:closedgeods} regarding
closed geodesics of $\overline M$ can be reobtained (in the Riemannian case).
The harmonic variational problem is also related to the CMC problem described in Subsection~\ref{subsec:cmc}.
Namely, an isometric immersion $\phi\colon M\to\overline M$ is minimal if and only if it is harmonic.
In addition, setting $\overline M=\mathds R$, harmonic maps are simply harmonic
functions on $M$; and setting $\overline M=\mathds S^1$,
harmonic maps are canonically identified with the harmonic
$1$-forms on $M$ with integral periods.
\end{rem}

Henceforth, we assume compactness of the source manifold $M$ to use the
classic variational characterization of harmonic maps.
Let $\mathfrak M$ be
the Banach manifold $\mathcal C^{2,\alpha}(M,\overline M)$ consisting of all
maps $\phi\colon M\to\overline M$ that satisfy the $\mathcal C^{2,\alpha}$-H\"older condition.
Let $\Lambda$ be the open subset of the Banach space of symmetric $(0,2)$-tensors of class $\mathcal C^k$ on $M$, with $k\ge3$,
consisting of all positive-definite tensors, i.e., elements of $\Lambda$ are Riemannian metric tensors of class $\mathcal C^k$
on $M$. Set $\mathfrak f\colon \mathfrak M\times\Lambda\to\mathds R$,
\[\mathfrak f(\phi,\mathbf g)=\tfrac12\int_M\Vert\mathrm d\phi(x)\Vert_{HS}^2\,\mathrm{vol}_{\mathbf g},\]
where $\mathrm{vol}_{\mathbf g}$ is the volume form (or density)
of $\mathbf g$ and $\Vert\mathrm d\phi(x)\Vert_{HS}$ is
the Hilbert-Schmidt norm of the linear map $\mathrm d\phi(x)$.
For a given $\mathbf g_0\in\Lambda$, critical points of the map $\phi\mapsto\mathfrak f(\phi,\mathbf g_0)$ are precisely
the $(\mathbf g_0,\overline{\mathbf g})$--harmonic maps $\phi\colon M\to\overline M$.
For $\phi\in\mathfrak M$, the tangent space $T_\phi\mathfrak M$ is identified
with the Banach space $\mathcal C^{2,\alpha}(\phi^*T\overline{M})$ of all
$\mathcal C^{2,\alpha}$-H\"older vector fields along $\phi$. Given a such $V\in T_\phi\mathfrak M$,
the derivative $\partial_1\mathfrak f(\phi,\mathbf g)V$ is given by:
\allowdisplaybreaks
\begin{multline}\label{eq:firstderharm}
\partial_1\mathfrak f(\phi,\mathbf g)V=
\int_M \mathrm{tr}\left(\mathrm d\phi^*\overline{\nabla}V\right) \,\mathrm{vol}_\mathbf g\\=
\int_M\big[\mathrm{div}\big(\mathrm d\phi^*(V)\big)
-\overline{\mathbf g}\big(\Delta_{\mathbf g,\overline{\mathbf g}}(\phi),V\big)\big]\,
\mathrm{vol}_\mathbf g\stackrel{\text{Stokes}}=
-\int_M\overline{\mathbf g}\big(\Delta_{\mathbf g,\overline{\mathbf g}}(\phi),V\big)\,\mathrm{vol}_\mathbf g,
\end{multline}
where the trace is meant on the entries $\mathrm d\phi^*(\cdot)\overline{\nabla}_{(\cdot)}V$.

\begin{defin}
The correspondent
{\em Jacobi operator} $J$ along a $(\mathbf g,\overline{\mathbf g})$-harmonic map $\phi$ is the linear differential operator:
\begin{equation}\label{eq:harmjacobiop}
J_\phi(V)=-\Delta V+\mathrm{tr}\left(\overline{R}(\mathrm d\phi(\cdot),V)\mathrm d\phi(\cdot)\right),
\end{equation}
defined in $\mathcal C^{2,\alpha}(\phi^*T\overline{M})$. Here $\overline R$ is the
curvature tensor of $\overline{\mathbf g}$, and
$\Delta V$ is a vector field along $\phi$ uniquely defined by
\begin{equation}\label{eq:deltav}
\overline{\mathbf g}(\Delta V,W)=\mathrm{div}(\overline{\nabla}V^*)W -\overline{\mathbf g}(\overline{\nabla}V,\overline{\nabla}W),
\quad W\in \mathcal C^{2,\alpha}(\phi^*T\overline{M})
\end{equation}
i.e., $\Delta V(x)=\sum_i\left(\overline{\nabla}_{e_i}\overline{\nabla}V\right)e_i$,
where $(e_i)_i$ is an orthonormal basis of $T_xM$.
\end{defin}

\begin{defin}
A vector field $V$ that satisfies $J_\phi(V)=0$ will be called a {\em Jacobi field}.
\end{defin}

Observe that if $K$ is a Killing vector field, then $J_\phi(K\circ\phi)=0$.

\begin{defin}
A $(\mathbf g,\overline{\mathbf g})$-harmonic map $\phi\colon M\to\overline M$
is said to be {\em nondegenerate} if all Jacobi fields along $\phi$ are of the form $K\circ\phi$, where $K$ is Killing.
\end{defin}

Let $G$ be the isometry group $\mathrm{Iso}(\overline M,\overline{\mathbf g})$ of the target manifold,
acting on $\mathfrak M$ by left-composition. Clearly, the functional $\mathfrak f$
is invariant in the first variable under this action.\footnote{%
One should observe that the harmonic map functional is also invariant under the action of the isometry group
$\mathrm{Iso}(M,\mathbf g)$ of the source manifold $(M,\mathbf g)$, which acts by right-composition
in the space of maps from $M$ to $\overline M$. However, equivariance with respect to such action will not
be considered here. Namely, observe that, as the metric $\mathbf g$ varies, then clearly also the group $\mathrm{Iso}(M,\mathbf g)$
varies; thus, in order to deal with such equivariance, a formulation of the equivariant implicit function theorem for varying groups
is needed. The assumption of equivariant nondegeneracy in Theorem B restricts the result to the case
where $(M,\mathbf g)$ has discrete isometry group or, more generally, when given any Killing field $K$ of $(M,\mathbf g)$, then
the field $\mathrm d\phi(K)$ along $\phi$ is the restriction to $\phi(M)$ of some Killing field $\overline K$ of $(\overline M,\overline{\mathbf g})$.}
Using results from \cite{palais}, it is possible
to prove that this action is smooth, since it is given by left-composition with smooth maps (see also \cite{PicTau} for the non-compact case).
As a consequence, part of the technical arguments in Theorem~\ref{thm:mainimplicitfunctiontheorem} to deal with
low regularity assumptions are not necessary in this context.

\begin{defin}
Two harmonic maps $\phi_1$ and $\phi_2$ are called {\em geometrically equivalent} if
they are in the same $\mathrm{Iso}(\overline M,\overline{\mathbf g})$-orbit, i.e.,  if there exists an isometry $\psi\colon \overline M\to\overline M$ such that $\phi_2=\psi\circ\phi_1$.
\end{defin}

We are now ready for:

\begin{proof}[Proof of Theorem B]
All assumptions of Theorem~\ref{thm:mainimplicitfunctiontheorem} are satisfied by the
harmonic maps problem, using the following objects:
\begin{itemize}
\item $\mathfrak M'$ coincides with $\mathfrak M=\mathcal C^{2,\alpha}(M,\overline M)$;

\item $\mathcal E$ is the mixed vector bundle whose fiber $\mathcal E_\phi$ is $\mathcal C^{0,\alpha}(\phi^*T\overline M)$,
the Banach space of all $\mathcal C^{0,\alpha}$-H\"older vector fields along $\phi$,
endowed with the topology $\mathcal C^{2,\alpha}$ on the base and $\mathcal C^{0,\alpha}$ on the fibers;

\item $\mathfrak i$ is the inclusion, and $\mathfrak j$ is induced by the $L^2$-pairing that uses the inner product induced
by $\overline{\mathbf g}$,  and integrals taken with respect to the
volume form (or density) of some fixed auxiliary\footnote{%
Note that one cannot use the volume form of $\mathbf g$ in order to define $\mathfrak j$, because this metric
is \emph{variable} in the problem that we are considering.} Riemannian metric $\mathbf g_*$ on $M$;

\item given $\phi\colon M\to\overline M$ of class $\mathcal C^{2,\alpha}$ and a Riemannian metric tensor $\mathbf g$ on $M$,
$\delta\mathfrak f(\phi,\mathbf g)$ is given by $-\zeta_\mathbf g\cdot\Delta_{\mathbf g,\overline{\mathbf g}}(\phi)$, where
$\zeta_\mathbf g\colon M\to\mathds R^+$ is the positive $\mathcal C^{k}$ function satisfying $\zeta_\mathbf g\cdot\mathrm{vol}_{\mathbf g_*}=\mathrm{vol}_{\mathbf g}$, see
\eqref{eq:deflaplacian} and \eqref{eq:firstderharm};

\item $\mathcal Y=T\mathfrak M$, $\kappa$ is the identity map and $\widetilde\jmath$ is induced by the $L^2$-pairing, as above;

\item identifying the Lie algebra $\mathfrak g$ with the space of (complete) Killing vector fields
on $(\overline M,\overline{\mathbf g})$, for $\phi\in\mathfrak M$, the map $\mathrm d\beta_\phi(1)\colon\mathfrak g\to T_\phi\mathfrak M$
associates to a Killing vector field $\overline{K}$ the vector field $\overline{K}\circ\phi$ along $\phi$;

\item given a $(\mathbf g,\overline{\mathbf g})$-harmonic map $\phi\colon M\to\overline M$, the vertical projection
of the derivative $\partial_1(\delta\mathfrak f)(\phi,\mathbf g)$ is identified with $\zeta_\mathbf g\cdot J_\phi$, where $J_\phi$ is the
Jacobi operator  in \eqref{eq:harmjacobiop}. This is an elliptic second order partial differential operator and
$\zeta_\mathbf g\cdot J_\phi\colon \mathcal C^{2,\alpha}(\phi^*T\overline M)\to\mathcal C^{0,\alpha}(\phi^*T\overline M)$ is a
Fredholm operator of index zero, see \cite[Thm~1.1, (2)]{Whi2}.\qedhere
\end{itemize}
\end{proof}

\subsection{Closed geodesics}\label{subsec:closedgeods}

Let $M$ be an arbitrary manifold, let $\mathfrak M$ be the Banach manifold $\mathcal C^2(\mathds S^1,M)$
consisting of all closed curves of class $\mathcal C^2$ in $M$, let $\mathcal B$ be a
Banach space of symmetric $(0,2)$-tensors of class $\mathcal C^k$ on $M$, with $k\ge3$, and
let $\Lambda$ denote an open subset of $\mathcal B$ consisting of tensors that are everywhere nondegenerate
on $M$. Thus, elements of $\Lambda$ are (pseudo-)Riemannian metric tensors on $M$.
We also fix an auxiliary Riemannian metric $\mathbf g_R$ on $M$; this metric induces a positive definite
inner product and a norm $\|\cdot\|_R$ on each tangent and cotangent space to $M$, and on all tensor products of
these spaces. This will be used implicitly throughout whenever our constructions require the use
of a norm or of an inner product of tensors.\footnote{%
These norms can be used, e.g., to give a simple construction of the Banach space $\mathcal B$. Consider $\nabla^R$ the Levi--Civita
connection of $\mathbf g_R$. Then $\mathcal B$ may be taken as the
space of $(0,2)$-tensors $s$ of class $\mathcal C^k$ on $M$ that are \emph{$\mathbf g_R$-bounded}, i.e.,
such that $\|s\|_{\mathcal B}=\max_{1\leq i\leq k}\left\{\sup_{x\in M}\left\|(\nabla^R)^i s(x)\right\|_R\right\}<+\infty.$
}
Given a (pseudo-)Riemannian metric tensor $\mathbf g$ on $M$,
let us denote by $\mathrm T_\mathbf g$ the $\mathbf g_R$-symmetric $(1,1)$-tensor on $M$
defined by:
\begin{equation}\label{eq:def11tensor}
\mathbf g=\mathbf g_R(\mathrm T_\mathbf g\cdot,\cdot).
\end{equation}

Consider the smooth function $\mathfrak f\colon \mathfrak M\times\Lambda\to\mathds R$ given by:
\[\mathfrak f(\gamma,\mathbf g)=\tfrac12\int_{\mathds S^1}\mathbf g(\gamma',\gamma')\,\mathrm d\theta;\]
for a given $\mathbf g_0\in\Lambda$, the critical points of the map $\gamma\mapsto\mathfrak f(\gamma,\mathbf g_0)$
are precisely the periodic $\mathbf g_0$-geodesics on $M$. For $\gamma\in\mathfrak M$, the tangent
space $T_\gamma\mathfrak M$ is identified with the Banach space of all periodic
vector fields $V$ of class $\mathcal C^2$
along $\gamma$. Given such a $V\in T_\gamma\mathfrak M$, recall that the derivative $\partial_1\mathfrak f(\gamma,\mathbf g)V$
is given by:
\begin{equation}\label{eq:firstderivativegeo}
\partial_1\mathfrak f(\gamma,\mathbf g)V=\int_{\mathds S^1}\mathbf g\big(\gamma',\tfrac{\mathrm D^\mathbf g}{\mathrm d\theta}V\big)\,\mathrm d\theta,
\end{equation}
where $\frac{\mathrm D^{\mathbf g}}{\mathrm d\theta}$ is the
covariant derivative operator along $\gamma$ relative to the Levi-Civita connection $\nabla^\mathbf g$ of $\mathbf g$.

\begin{defin}
If $\gamma$ is a $\mathbf g$-geodesic, the \emph{Jacobi operator}
$J$ along $\gamma$ is the linear differential operator:
\begin{equation}\label{eq:Jacobioperator}
J(V)=\left(\frac{\mathrm D^\mathbf g}{\mathrm d\theta^2}\right)^2V+R^\mathbf g(\gamma',V)\gamma',
\end{equation}
defined in the space of $\mathcal C^2$-vector fields $V$ along $\gamma$.
Here $R^\mathbf g$ is the curvature tensor of the Levi-Civita connection of $\mathbf g$.
A Jacobi field along $\gamma$ is a vector field $V$ satisfying $J(V)=0$.
\end{defin}

\begin{defin}
A closed $\mathbf g$-geodesic $\gamma$ is said to be
\emph{nondegenerate} if the only periodic Jacobi fields along $\gamma$ are (constant) multiples
of the tangent field $\gamma'$.
\end{defin}

\begin{rem}
Nondegeneracy of all closed geodesics (including iterates) is a
\emph{generic} property in the set of (pseudo-)Riemannian metric tensors $\mathbf g$, see \cite{bettiol,bumpy}.
\end{rem}

\begin{defin}
Let $G$ be the circle $\mathds S^1$, acting on $\mathfrak M$ by rotation, i.e.,
by right-composition. This action is only continuous (and not differentiable),
but each $g\in G$ gives a diffeomorphism of $\mathfrak M$.
The stabilizer of every non-constant closed curve $\gamma$ in $M$
is a finite cyclic subgroup of $\mathds S^1$. When such stabilizer is trivial,
we say that $\gamma$ is \emph{prime}, i.e., it is not the iterate of some other closed curve
in $M$. If $n>1$ is the order of the stabilizer of a curve $\gamma$,
then $\gamma$ is the $n$-fold iterate of some prime closed curve
on $M$. Two closed curves $\gamma_1$ and $\gamma_2$ on $M$ belong to the
same $\mathds S^1$-orbit if and only if:
\begin{itemize}
\item[(a)] $\gamma_1(\mathds S^1)=\gamma_2(\mathds S^1)$;
\item[(b)] $\gamma_1$ and $\gamma_2$ have stabilizers of the same order.
\end{itemize}
When (a) and (b) are satisfied, we will say that $\gamma_1$ and $\gamma_2$ are \emph{geometrically
equivalent}.
\end{defin}

We are now ready for:

\begin{proof}[Proof of Theorem C]
All assumptions of Theorem~\ref{thm:mainimplicitfunctiontheorem} are satisfied by the
harmonic maps problem, using the following objects:
\begin{itemize}
\item $\mathfrak M'$ is the set $\mathcal C^3(\mathds S^1,M)$, endowed with the $\mathcal C^3$-topology;
\item $\mathcal E$ is the mixed vector bundle whose fiber $\mathcal E_\gamma$ is the Banach space
of all periodic continuous vector fields along $\gamma$, endowed with the topology $\mathcal C^2$ on the base
and $\mathcal C^0$ on the fibers;
\item $\mathfrak i$ is the inclusion, and $\mathfrak j$ is induced by the $L^2$-pairing (this uses the inner product given by $\mathbf g_R$);
\item $\mathcal Y$ is the mixed vector bundle whose fiber $\mathcal Y_\gamma$ is the Banach space
of all periodic $\mathcal C^1$-vector fields along $\gamma$, endowed with the topology $\mathcal C^2$ on the base
and $\mathcal C^1$ on the fibers;
\item $\widetilde{\jmath}$ is induced by the $L^2$-pairing (this uses the inner product given by $\mathbf g_R$);
\item $\kappa$ is the inclusion;
\item using the identification $\mathfrak g\cong\mathds R$, $\mathrm{Lin}(\mathfrak g,T\mathfrak M)\cong T\mathfrak M$
and $\mathrm{Lin}(\mathfrak g,\mathcal Y)\cong\mathcal Y$, for $\gamma\in\mathfrak M'$,
the map $\mathrm d\beta_\gamma(1)$ is the element $\gamma'\in T\mathfrak M$;
\item the map $\kappa\circ[\mathrm d\beta_\gamma(1)]$ has the same expression of $\mathrm d\beta_\gamma(1)$, where
now $\gamma\in\mathfrak M$ and $\gamma'\in\mathcal Y$;
\item the map $\delta\mathfrak f$ is defined by $\delta\mathfrak f(\gamma,\mathbf g)=
-\mathrm T_\mathbf g\big(\frac{\mathrm D^{\mathbf g}}{\mathrm d\theta}\gamma'\big)$, where
$\mathrm T_\mathbf g$ is defined in \eqref{eq:def11tensor}.
Note that $\frac{\mathrm D^{\mathbf g}}{\mathrm d\theta}\gamma'$ is a continuous vector field along $\gamma$, and:
\begin{multline*}
\mathfrak j_\gamma\big(\delta\mathfrak f(\gamma,\mathbf g)\big)V=\int_{\mathds S^1}\mathbf g_R\big(\delta\mathfrak f(\gamma,\mathbf g),V\big)\,\mathrm d\theta=
-\int_{\mathds S^1}\mathbf g_R\big(\mathrm T_\mathbf g(\tfrac{\mathrm D^{\mathbf g}}{\mathrm d\theta}\gamma'),V\big)\,\mathrm d\theta\\=
-\int_{\mathds S^1}\mathbf g\big(\tfrac{\mathrm D^{\mathbf g}}{\mathrm d\theta}\gamma',V\big)\,\mathrm d\theta=
\int_{\mathds S^1}\mathbf g\big(\gamma',\tfrac{\mathrm D^\mathbf g}{\mathrm d\theta}V\big)\,\mathrm d\theta=
\partial_1\mathfrak f(\gamma,\mathbf g)V
\end{multline*}
\item The derivative $\partial_1(\delta\mathfrak f)$ is given by:
\[\partial_1(\delta\mathfrak f)(\gamma,\mathbf g)V=-\big(\nabla_V^\mathbf g\mathrm T_\mathbf g\big)\big(\tfrac{\mathrm D^{\mathbf g}}{\mathrm d\theta}\gamma'\big)-
\mathrm T_\mathbf g\big(J(V)\big),\] where $V\in T_\gamma\mathfrak M$
and $J$ is the Jacobi operator \eqref{eq:Jacobioperator}.

The operator $J$ acting on the space of periodic fields of class $\mathcal C^2$
along $\gamma$ and taking values in the space of periodic continuous vector fields along $\gamma$ is a Fredholm
operator of index $0$, as it is a compact perturbation of an isomorphism.
Since the composition on the left with $\mathrm T_\mathbf g$ is an isomorphism, it follows that the operator
$V\mapsto \mathrm T_\mathbf g\big(J(V)\big)$ is a Fredholm operator of index $0$
from the space of periodic fields of class $\mathcal C^2$
along $\gamma$ to the space of periodic continuous vector fields along $\gamma$.
The operator $V\mapsto-\big(\nabla_V^\mathbf g\mathrm T_\mathbf g\big)\big(\frac{\mathrm D^{\mathbf g}}{\mathrm d\theta}\gamma'\big)$
from the space of $\mathcal C^2$-vector fields to the space of $\mathcal C^0$-vector fields is compact, as it is
continuous relatively to the $\mathcal C^0$-topology, and the inclusion $\mathcal C^2\hookrightarrow\mathcal C^0$ is compact.
Hence, $\partial_1(\delta\mathfrak f)(\gamma,\mathbf g)$ is Fredholm of index $0$.

\item If $\gamma\in\mathfrak M'$, then the orbit $\mathds S^1\cdot\gamma$ is a $\mathcal C^1$-submanifold
of $\mathfrak M$ which is diffeomorphic to $\mathds S^1$.
The tangent space $T_\gamma(\mathds S^1\cdot\gamma)\subset T_\gamma\mathfrak M$
is spanned by the tangent field $\gamma'$. Nondegeneracy of a critical orbit thus corresponds to
the nondegeneracy of the closed geodesic.\qedhere
\end{itemize}
\end{proof}

\subsection{CMC hypersurfaces}\label{subsec:cmc}
Using our main abstract result, we now prove Theorem D, which is an improvement of Proposition~\ref{thm:implfunctthmCMC}. More precisely, we employ a version of Theorem~\ref{thm:mainimplicitfunctiontheorem} for local actions, see Theorem~\ref{thm:localmainimplicitfunctiontheorem}. We need a technical assumption concerning the existence of an \emph{invariant volume functional}
around a given CMC embedding $x\colon M\hookrightarrow\overline M$. This will be a volume functional invariant
under left-compositions with isometries of the ambient space, see Definition~\ref{thm:definvariantvolumes}.
Examples where this assumption is satisfied are discussed in Appendix~\ref{sec:invariantvolumes}.
We stress that this assumption is indeed necessary, see Example~\ref{exa:counterexamplelackofvolume}.

\begin{proof}[Proof of Theorem D]
Consider the set $\mathfrak E(M,\overline M)$ of all \emph{unparameterized embeddings} of class $\mathcal C^{2,\alpha}$ of
$M$ into $\overline M$, i.e., the set of congruence classes of $\mathcal C^{2,\alpha}$ embeddings  $y\colon M\hookrightarrow\overline M$.
Such a set does not have a natural global differentiable structure, but it admits an atlas of charts that make it an infinite-dimensional topological manifold modeled on the Banach space $\mathcal C^{2,\alpha}(M)$, see \cite{AliPic10}.
Given a smooth embedding $y\colon M\to\overline M$, nearby congruence classes of embeddings are parameterized by sections of
the normal bundle of $y$, using the exponential map of $(\overline M,\overline{\mathbf g})$.
We will identify congruence classes of embeddings near $x$ with functions belonging to a neighborhood of $0$ in the Banach space
$\mathcal C^{2,\alpha}(M)$; for this identification the transversal orientation of $x(M)$ is used.
Let $\mathfrak M$ be a sufficiently small neighborhood of $x$ in $\mathfrak E(M,\overline M)$, identified
with a neighborhood of $0$ in the space $\mathcal C^{2,\alpha}(M)$.

Consider the isometry group $G=\mathrm{Iso}(\overline M,\overline{\mathbf g})$ of the ambient manifold. There is a \emph{local action} (see Appendix~\ref{sec:localactions}) of $G$ on $\mathfrak M$, defined as follows. If $y\colon M\hookrightarrow\overline M$ is an embedding near $x$ and $\phi$ is an isometry
of  $(\overline M,\overline{\mathbf g})$, then the action of $\phi$ on (the congruence class of) $y$ is given by the
(congruence class of the) left-composition $\phi\circ y$.
The domain of this action consists of pairs $(\phi,y)$ such that $\phi\circ y$ belongs to $\mathfrak M$; the axioms of
local actions are readily verified for this map.
The local action of $G$ on the set of unparameterized embedding is continuous (see \cite{AliPic10}), but the action is
differentiable only on the dense subset $\mathfrak M'$ of $\mathfrak M$ consisting of congruence classes of
embeddings of class $\mathcal C^{3,\alpha}$. The orbit of each of
these elements is a $\mathcal C^1$-submanifold of $\mathcal C^{2,\alpha}(M)$.

Given an embedding $y\colon M\hookrightarrow\overline M$,
denote by $\mathcal A(y)$ the volume of $M$ relatively to the volume form of the pull-back metric
$y^*(\overline{\mathbf g})$, and consider an invariant volume functional $\mathcal V$ defined in a neighborhood of $x$ in $\mathcal C^1(M,\overline M)$.
The values $\mathcal A(y)$ and $\mathcal V(y)$ do not depend on the parameterization of $y$, and
$\mathcal A$ and $\mathcal V$ define functions on $\mathfrak M$ that are smooth in every local chart, see
\cite{AliPic10} for details.
Finally, let $\Lambda$ be an open interval of $\mathds R$ containing $H_0$, and define
\begin{equation*}
\begin{aligned}
\mathfrak f &\colon \mathfrak M\times\Lambda \longrightarrow\mathds R \\
\mathfrak f(y,\lambda) &:=\mathcal A(y)+m\lambda \mathcal V(y)
\end{aligned}
\end{equation*}
It is well known that $\partial_1\mathfrak f(y,\lambda)=0$ if and only if $y$ is a CMC embedding
with mean curvature equal to $\lambda$. The second variation of $\mathfrak f(\,\cdot,H_0)$ at $x$ is identified with the
Jacobi operator $J_x$ in \eqref{eq:JacobiCMC}. In particular, $J_x\colon \mathcal C^{2,\alpha}(M)\to\mathcal C^{0,\alpha}(M)$ is a Fredholm operator of index 0, see~\cite[Thm~1.1, (2)]{Whi2}.
Since $\mathcal A$ and $\mathcal V$ are invariant under the local action of left-composition with elements of the isometry group
$G=\mathrm{Iso}(\overline M,\overline{\mathbf g})$, then so is the function $\mathfrak f$. Recall that $G$ is a Lie group, and is compact when $\overline M$ is compact.

The desired result now follows as a direct application of the equivariant implicit function theorem for local actions,
see Theorem~\ref{thm:localmainimplicitfunctiontheorem}; the objects described in the axioms (A), (B), (D) and (F) are
defined as follows for the CMC variational problem:
\begin{itemize}
\item $\mathcal E$ is the Banach space $\mathcal C^{0,\alpha}(M)$;
\item $\mathfrak i$ is the inclusion $\mathcal C^{2,\alpha}\hookrightarrow\mathcal C^{0,\alpha}$ and
$\mathfrak j$ is induced by the $L^2$-pairing $(f,g)\mapsto\int_Mf\cdot g\,\mathrm{vol}_\mathbf g$;
\item $\mathcal Y$ is the Banach space $\mathcal C^{1,\alpha}(M)$;
\item $\widetilde\jmath$ is induced by the $L^2$-pairing;
\item $\kappa$ is the inclusion;
\item identifying the Lie algebra $\mathfrak g$ with the space of (complete) Killing vector fields
on $(\overline M,\overline{\mathbf g})$, for $y\in\mathfrak M'$, the map $\mathrm d\beta_y(1)\colon \mathfrak g\to T_y\mathfrak M$
associates to a Killing vector field $\overline{K}$ the orthogonal component of $\overline{K}$ along $y$;
\item given a $\mathcal C^{2,\alpha}$-embedding $y$,  $\delta\mathfrak f(y,\lambda)$ is the mean curvature function
of $y$ (which is a $\mathcal C^{0,\alpha}$-function on $M$);
\item $\partial_1(\delta\mathfrak f)(x,H_0)$ is identified with the Jacobi operator $J_x$.
\end{itemize}
The assumptions of Theorem~\ref{thm:mainimplicitfunctiontheorem}/Theorem~\ref{thm:localmainimplicitfunctiontheorem} are easily verified, concluding the proof.
\end{proof}

The following examples show that neither the assumption on the existence of an invariant volume functional
nor the assumption on the transversal orientation in the case of minimal embeddings can be omitted
in Theorem D.
\begin{example}

\label{exa:counterexamplelackofvolume}
Consider $M=\mathds S^1$ and $\overline M=\mathds S^1\times\mathds S^1$ is the $2$-torus
endowed with the flat metric. The embedding $x\colon \mathds S^1\to\mathds S^1\times\mathds S^1$ given by $x(z)=(z,1)$, $z\in\mathds S^1$,
is obviously minimal (i.e., a geodesic). It is also easy to see that such embedding is nondegenerate, i.e., every periodic Jacobi field
along $x$ is the restriction of a Killing vector field. However, near $x$ there exists no embedding of $\mathds S^1$ into $\mathds S^1\times\mathds S^1$
with constant geodesic curvature different from zero. Namely, every constant geodesic curvature embedding should be the projection on
$\mathds S^1\times\mathds S^1$ of a circle in the plane $\mathds R^2$; such a projection is a curve with trivial homotopy class, hence
it cannot be close to $x$ in the $\mathcal C^1$-topology. Observe that, in this example, the image of $x$ is not contained in any open susbet
of $\mathds S^1\times\mathds S^1$ with trivial first cohomology space, and there exists no volume functional defined in
a neighborhood of $x$ in $\mathcal C^1(\mathds S^1,\mathds S^1\times\mathds S^1)$ which is
invariant under isometries of $\mathds S^1\times\mathds S^1$.
\end{example}

\begin{example}\label{thm:exaprojective}
We observe that the transverse orientability is a necessary condition in Theorem D. Namely, such condition is closed (and also open) relatively
to the $\mathcal C^1$-topology; thus, if $(x_H)_{H\in\left]-\varepsilon,\varepsilon\right[}$ is a continuous
family of CMC embeddings, such that each $x_H$ has mean curvature $H$, then $x_0$ must be transversely oriented.

For instance, consider the real projective plane $\mathds R P^2$ with the standard metric, and  the minimal (i.e., geodesic) embedding $x_0:\mathds S^1\hookrightarrow\mathds R P^2$ obtained by projecting
in $\mathds RP^2$ a minimal geodesic between two antipodal points in $\mathds S^2$. This is a nondegenerate
minimal embedding, which is not transversely oriented. The only CMC immersions of $\mathds S^1$ in
$\mathds S^2$ are \emph{parallel} to the equator. The corresponding immersions obtained in the
quotient $x_H\colon \mathds S^1\hookrightarrow\mathds R P^2$ are \emph{not} $\mathcal C^1$-close to $x_0$ since, for instance, the length
of $x_h$ tends to twice the length of $x_0$ as $H$ goes to $0$.
\end{example}

An alternative form of stating Theorem D uses the notion of \emph{rigidity} for a path
of CMC embeddings.

\begin{defin}\label{thm:deflocalrigidity}
Given a $1$-parameter family $x_s\colon M\hookrightarrow\overline M$, $s\in[a,b]$, of
CMC embeddings, we say that the family $X=\{x_s\}_{s\in[a,b]}$ is \emph{rigid}
if there exists an open neighborhood $\mathcal U$ of $X$ in $\mathcal C^{2,\alpha}(M,\overline M)$ such that any CMC embedding $x\colon M\hookrightarrow\overline M$
in $\mathcal U$ is isometrically congruent to some $x_s$.
We say that the family is \emph{locally rigid} at $s_0\in[a,b]$ if there exists
$\varepsilon>0$ such that, setting $I=[a,b]\cap[s_0-\varepsilon,s_0+\varepsilon]$,
the family $\{x_s\}_{s\in I}$ is rigid.
\end{defin}

\begin{cor}\label{thm:corrigidity}
Let $x_s\colon M\hookrightarrow\overline M$, $s\in[a,b]$, be a $\mathcal C^1$-family of
CMC embeddings, denote by $\mathcal H(s)$ the mean curvature of $x_s$, and let $s_0\in[a,b]$ be such that:
\begin{itemize}
\item $x_{s_0}$ is nondegenerate;
\item there exists an invariant volume functional in a $\mathcal C^1$-neighborhood of $x_{s_0}$;
\item $\mathcal H'(s_0)\ne0$.
\end{itemize}
Then, $X=\{x_s\}_{s\in[a,b]}$ is locally rigid at $s_0$.
\end{cor}

\begin{proof}
The assumption $\mathcal H'(s_0)\ne0$ implies the existence of a $\mathcal C^1$-function \[\left]\mathcal H(s_0)-\varepsilon,\mathcal H(s_0)+\varepsilon\right[\ni H\longmapsto
s(H)\in\left]s_0-\delta,s_0+\delta\right[,\] with $\varepsilon,\delta>0$ small enough, such that $\mathcal H\big(s(H)\big)=H$ for all $H$.
Apply Theorem D to $x=x_{s_0}$, obtaining a new path $H\mapsto x_H$ of CMC embeddings. Note that $x_{s_0}$ must be transversely oriented, even in the case $\mathcal H(s_0)=0$;
namely, by the assumption $\mathcal H'(s_0)\ne0$, it follows that $\mathcal H$ is not constant in any neighborhood
of $s_0$, and thus $x_{s_0}$ is limit of transversely oriented embeddings.
By part (b) of Theorem D,  $x_{s(H)}$ must be isometrically congruent to $X_H$ for all $H$ near $H(s_0)$,
and the local rigidity follows readily.
\end{proof}

\subsubsection*{CMC embeddings of manifolds with boundary}
A result totally analogous to Theorem D holds in the case of codimension one CMC embeddings $x\colon M\hookrightarrow\overline M$ of manifolds $M$ with boundary $\partial M$. In this situation, one is interested in variations of $x$ that fix the boundary, and the corresponding infinitesimal variations are Jacobi fields that vanish on $\partial M$. We have the corresponding:

\begin{defin}
If $\partial M\neq\emptyset$, a CMC embedding $x\colon M\hookrightarrow\overline M$ is
\emph{nondegenerate} if every Jacobi field $f$ along $x$ that vanishes on $\partial M$ is of the form
$f=\overline{\mathbf g}(\overline{K},\vec n_x)$ for
some Killing field $\overline{K}$ of $(\overline M,\overline{\mathbf g})$ tangent to $x(\partial M)$.
\end{defin}

If $x$ is nondegenerate in this sense, then the implicit function theorem gives the
existence of a variation $(x_H)_{H\in\left]H_0-\varepsilon,H_0+\varepsilon\right[}$ of $x$ by CMC embeddings
$x_H\colon M\hookrightarrow\overline M$ such that $x_H\vert_{\partial M}=x\vert_{\partial M}$ for all $H$.
The proof of this fixed boundary version is totally analogous to that of Theorem D, \emph{mutatis mutandis}.
It is required the existence of a volume functional defined in a $\mathcal C^1$-neigborhood of $x$ in the
set of embeddings $y\colon M\hookrightarrow\overline M$ with fixed boundary, i.e., $y(\partial M)=x(\partial M)$, and
which is invariant under isometries of $(\overline M,\overline{\mathbf g})$ that preserve $x(\partial M)$.
Note that the group of such isometries is always compact, because the action of the isometry group is
proper and $\partial M$ is compact. In this case, invariant volume functionals can be obtained from
invariant primitives of the volume form, using an averaging procedure, see Appendix~\ref{sec:invariantvolumes}.

\begin{rem}
As to the variational framework, in the non-empty boundary case $\mathfrak M$ is the manifold
of fixed boundary unparameterized embeddings of $M$ into $\overline M$ of class $\mathcal C^{2,\alpha}$, which is modeled on the
Banach space $\mathcal C^{2,\alpha}_0(M,\mathds R)=\big\{f\in\mathcal C^{2,\alpha}(M,\mathds R):f\vert_{\partial M}\equiv0\big\}$.
Note that, when $M$ has boundary, the Jacobi operator $J_x\colon \mathcal C^{2,\alpha}_0(M,\mathds R)\to\mathcal C^0(M,\mathds R)$ is Fredholm of index $0$.
\end{rem}

\subsubsection*{Natural extensions of the CMC implicit function theorem}
Theorem D extends naturally to more general situations
involving hypersurfaces that are stationary for a parametric elliptic functional with a volume constraint,
like for instance hypersurfaces with constant \emph{anisotropic mean curvature}, see \cite{KoiPal}.
Such extension is quite straightforward, and will not be discussed here.
It is also interesting to point out that Theorem D can be extended
to the case of CMC \emph{immersions}, rather than embeddings. The question here is to endow the set
of unparameterized immersions with a local differential structure based on the exponential map of
the normal bundle. This is possible in the case of the so-called \emph{free immersions}, i.e.,
immersions $x\colon M\to\overline M$ with the property that the unique diffeomorphism $\phi$ of $M$
satisfying $x\circ\phi=x$ is the identity. This is the case, for instance, when there is some point
in the image of $x$ whose inverse image consists of a single point of $M$. Details can be found in \cite{CerMasMic}.

\end{section}

\appendix
\begin{section}{Local actions}\label{sec:localactions}
It is useful to have a version of the equivariant implicit function theorem for local actions of Lie groups
on manifold. Once more, the paradigmatic example to keep in mind is the CMC embedding problem, in which
one has a local action of the isometry group of the target manifold on a neighborhood of $0$ of a Banach space,
see Subsection~\ref{subsec:cmc}. One observes that, given the local character of the result, the proof of
Theorem~\ref{thm:mainimplicitfunctiontheorem} carries over to the case where the action of the Lie group
$G$ is only locally defined, in a sense that will be clarified in this appendix.

\begin{defin}
Let $G$ be a Lie group and let $\mathfrak M$ be a topological manifold.
A \emph{local action} of $G$ on $\mathfrak M$ is a continuous map $\rho\colon \mathrm{Dom}(\rho)\subset G\times\mathfrak M\to\mathfrak M$,
defined on an open subset $\mathrm{Dom}(\rho)\subset G\times\mathfrak M$ containing $\{1\}\times\mathfrak M$ satisfying:
\begin{itemize}
\item[(a)] $\rho(1,x)=x$ for all $x\in\mathfrak M$;
\item[(b)] $\rho\big(g_1,\rho(g_2,x)\big)=\rho(g_1g_2,x)$ whenever both sides of the equality are defined, i.e., for all those $x\in\mathfrak M$ and $g_1,g_2\in G$
such that $(g_2,x)\in\mathrm{Dom}(\rho)$, $\big(g_1,\rho(g_2,x)\big)\in\mathrm{Dom}(\rho)$ and $(g_1g_1,x)\in\mathrm{Dom}(\rho)$.
\end{itemize}
\end{defin}

\begin{rem}
The particular case of actions is when the domain $\mathrm{Dom}(\rho)$ coincides with the entire $G\times\mathfrak M$.
\end{rem}

\begin{rem}
Local actions can be restricted, in the sense that, given any open subset $\mathcal A$ of $\mathrm{Dom}(\rho)$ containing
$\{1\}\times\mathfrak M$, then the restriction $\rho\vert_\mathcal A$ of $\rho$ to $\mathcal A$ is again a local action of $G$ on $\mathfrak M$.
\end{rem}

Given a local action $\rho$ of $G$ on $\mathfrak M$, for $g\in G$, let $\rho_g$ denote the map $\rho(g,\cdot)$, defined
on an open (possibly empty) set $\mathrm{Dom}(\rho_g)=\mathrm{Dom}(\rho)\cap\{g\}\times\mathfrak M$.
The following properties follow easily from the definition:
\begin{lem}\label{thm:lemlocalactions}
Let $\rho\colon \mathrm{Dom}(\rho)\subset G\times\mathfrak M\to\mathfrak M$ be a local action of $G$ on $M$. Then:
\begin{itemize}
\item[(i)] for all $g\in G$, the map $\rho_g\colon\rho_g^{-1}\big(\mathrm{Dom}(\rho_{g^{-1}})\big)\to\rho_{g^{-1}}^{-1}\big(\mathrm{Dom}(\rho_g)\big)$ is a homeomorphism;
\item[(ii)] the set $\big\{(g,x)\in G\times\mathfrak M:x\in\rho_g^{-1}\big(\mathrm{Dom}(\rho_{g^{-1}})\big)\big\}$ is an open subset that contains $\{1\}\times\mathfrak M$;
in particular:
\item[(iii)] for all $x\in\mathfrak M$, there exists an open neighborhood $U_x$ of $1$ in $G$ such that for all $g\in U_x$, $x\in\rho_g^{-1}\big(\mathrm{Dom}(\rho_{g^{-1}})\big)$.
\end{itemize}
\end{lem}
In view of (iii), one can define a map $\beta_x\colon \mathrm{Dom}(\beta_x)\subset G\to\mathfrak M$ on a neighborhood $\mathrm{Dom}(\beta_x)$ of $1$
in $G$, by $\beta_x(g)=\rho(g,x)$, cf. \eqref{eq:defbetax}. In particular, if $x\in\mathfrak M$ is such that $\beta_x$ is differentiable (at $1$),
then one has a well-defined linear map $\mathrm d\beta_x(1)\colon\mathfrak g\to T_x\mathfrak M$. A subset $C\subset\mathfrak M$ will be called $G$-invariant if $x\in C$
implies $\rho(g,x)\in C$ for all $g\in\mathrm{Dom}(\beta_x)$.

\begin{teo}\label{thm:localmainimplicitfunctiontheorem}
Theorem~\ref{thm:mainimplicitfunctiontheorem} holds when one replaces {\rm (A2)} with the assumption that a local
action $\rho$ of $G$ on $\mathfrak M$ is given, and {\rm (A3)} with the assumption that $\mathfrak f$ satisfies:
\begin{equation*}
\mathfrak f\big(\rho(g,x),\lambda\big)=\mathfrak f(x,\lambda),\quad\mbox{ for all }\,(g,x)\in\mathrm{Dom}(\rho).
\end{equation*}
In this situation, the conclusion is that there exist open subsets $\Lambda_0\subset\Lambda$ and $\mathfrak M_0\subset\mathfrak M$, with $\lambda_0\in\Lambda_0$ and
$x_0\in\mathfrak M_0$ and a  $\mathcal C^k$ map $\sigma\colon\Lambda_0\to\mathfrak M_0$ such that, for $(x,\lambda)\in\mathfrak M_0\times\Lambda_0$,
the identity $\partial_1\mathfrak f(x,\lambda)=0$
holds if and only if there exists $g\in G$, with $\big(\phi(\lambda),g\big)\in\mathrm{Dom}(\rho)$ such that $x=\rho\big(g,\sigma(\lambda)\big)$.
\end{teo}
\begin{proof}
The proof of Theorem~\ref{thm:mainimplicitfunctiontheorem} carries over to this case, with minor modifications.
One constructs a submanifold $S\subset\mathfrak M$ through $x_0$ with the property that $T_{x_0}S\oplus\mathrm{Im}\big(\mathrm d\beta_{x_0}(1)\big)=T_{x_0}\mathfrak M$,
and considers the restriction of $\mathfrak f$ to the product $S\times\Lambda$. The only question that needs an extra argument here is to show that the
set $\rho\big((G\times S)\cap\mathrm{Dom}(\rho)\big)$ is a neighborhood of $x_0$; this is the analogue of claim \eqref{itm:fact2}
in the proof of Theorem~\ref{thm:mainimplicitfunctiontheorem} (see page \pageref{itm:fact2}).
Once again, this follows as an application of Proposition~\ref{thm:nonemptyintersection}, used in the following setup:
$A$ is an open neighborhood of $1$ in $G$, $M$ is an open neighborhood of $x_0$ in $\mathfrak M$, these open subsets being chosen in such a way that
the product $A\times M$ be contained in the open set $\big\{(g,x)\in G\times\mathfrak M:x\in\rho_g^{-1}\big(\mathrm{Dom}(\rho_{g^-1})\big)\big\}$, see
part (ii) of Lemma~\ref{thm:lemlocalactions}. Set $N=\mathfrak M$, $P=S$, $a_0=1$ and $m_0=x_0$; the function $\chi$ is the restriction
of $\rho$ to $A\times M$. The conclusion of Proposition~\ref{thm:nonemptyintersection} says that for all $x\in M$, there exists $g\in A$ such that
$x\in\rho_g^{-1}\big(\mathrm{Dom}(\rho_{g^{-1}})\big)$ and $\rho(g,x)\in S$; since $\rho(g,x)\in\mathrm{Dom}(\rho_{g^-1})$, then $x=\rho\big(g^{-1},\rho(g,x)\big)\in\rho\big((G\times S)\cap\mathrm{Dom}(\rho)\big)$,
i.e., $\rho\big((G\times S)\cap\mathrm{Dom}(\rho)\big)$ contains the open subset $M\ni x_0$.
The rest of the proof of Theorem~\ref{thm:mainimplicitfunctiontheorem} can now be repeated \emph{verbatim}.
\end{proof}
\end{section}

\begin{section}{Invariant volume functionals}
\label{sec:invariantvolumes}
A technical assumption made in Theorem D concerns the existence of a generalized volume
functional $\mathcal V$ which is invariant under left-composition with isometries of $(\overline M,\overline{\mathbf g})$.
Let us consider a compact differentiable manifold $M$, possibly with boundary $\partial M$, and a Riemannian manifold $(\overline M,\overline{\mathbf g})$,
with $m=\mathrm{dim}(M)=\mathrm{dim}(\overline M)-1$.
\begin{defin}\label{thm:definvariantvolumes}
Let $\mathcal U$ be an open subset of embeddings $x\colon M\hookrightarrow\overline M$. An \emph{invariant volume functional} on $\mathcal U$
is a real valued function $\mathcal V\colon \mathcal U\to\mathds R$ satisfying:
\begin{itemize}
\item[(a)] $\mathcal V(x)=\int_Mx^*(\eta)$, where $\eta$ is an $m$-form defined on an open subset $U\subset\overline M$ such that
$\mathrm d\eta$ is equal to the volume form $\mathrm{vol}_{\overline{\mathbf g}}$ of $\overline{\mathbf g}$ in $U$;

\item[(b)] given $x\in\mathcal U$, for all $\phi$ isometry of $(\overline M,\overline{\mathbf g})$ sufficiently close to the identity
it is $\mathcal V(\phi\circ x)=\mathcal V(x)$.
\end{itemize}
If $M$ has boundary, the invariance property (b) is required to hold only for isometries $\phi$ near the identity, and preserving
$x(\partial M)$, i.e., $\phi\big(x(\partial M)\big)=x(\partial M)$.
\end{defin}
By (b), the generalized volume $\mathcal V(x)$ does not depend on the parameterization of $x$, i.e., $\mathcal V(x\circ\psi)=\mathcal V(x)$ for all diffeomorphisms
$\psi$ of $M$. Hence, $\mathcal V$ defines a smooth map (in local charts) in a neighborhood of $x$ in the set of unparameterized embeddings of $M$
into $\overline M$, see \cite{AliPic10}.

\begin{example}
Assume that the image $x(M)$ of $x$ is the boundary of a bounded open subset $\Omega$ of $\overline M$; note
that this is an open condition in the set of unparameterized embeddings.\footnote{%
If $x$ is transversely oriented, then such condition is equivalent to the fact that $x$
induces the null map in homology $H^1(M)\to H^1\big(\overline M,M\setminus x(M)\big)$.
In particular, the condition is stable by $\mathcal C^1$-small perturbations of $x$.}
For $y$ sufficiently close to $x$ in the $\mathcal C^1$-topology, one has $y(M)=\partial\Omega_y$ for some open bounded
subset $\Omega_y$ of $\overline M$. Setting $\mathcal V(y)=\mathrm{vol}(\Omega_y)$, i.e., the volume of this open bounded subset,
we have an invariant volume functional $\mathcal V$ (use Stokes' theorem).
In fact, when $x(M)=\partial\Omega$, any function $\mathcal V$ satisfying {\rm (a)} of Definition~\ref{thm:definvariantvolumes}
coincides with such volume functional, by Stokes' theorem.
\end{example}

\begin{example}
Assume that $\overline M$ is non-compact, so that its volume form $\mathrm{vol}_{\overline{\mathbf g}}$
 is exact, and assume that $G=\mathrm{Iso}(\overline M,\overline{\mathbf g})$
is compact. Let $\eta$ be a primitive of $\mathrm{vol}_{\overline{\mathbf g}}$, and set:
\[\eta^G=\int_G\phi^*(\eta)\,\mathrm d\phi,\]
the integral being taken relatively to a Haar measure of volume $1$ of $G$. Then, $\eta^G$ is a primitive of
$\mathrm{vol}_{\overline{\mathbf g}}$ which
is $G$-invariant. In particular, the functional $\mathcal V(x)=\int_Mx^*(\eta^G)$ is an invariant volume functional.
\end{example}
Using the compactness argument above, we obtain immediately:
\begin{cor}\label{thm:invvolumefunctboundary}
Assume that $M$ is a compact manifold with (non-empty) boundary $\partial M$, and that $\overline M$ is non-compact.
Given an embedding $x\colon M\hookrightarrow\overline M$,
there exists a volume functional $\mathcal V$ defined
in the set:
\[\Big\{y\colon M\hookrightarrow\overline M\ \text{embedding}:  y(\partial M)=x(\partial M)\Big\}\]
which is invariant under all isometries $\phi$ that preserve $x(\partial M)$.
\end{cor}
\begin{proof}
The action of the isometry group $G=\mathrm{Iso}(\overline M,\overline{\mathbf g})$ on $\overline M$ is proper;
since $\partial M$ is compact, the closed subgroup $G_0$ of $G$ consisting of isometries $\phi$ satisfying $\phi\big(x(\partial M)\big)= x(\partial M)$
is compact. If $\eta$ is a primite of $\mathrm{vol}_{\overline{\mathbf g}}$ in $\overline M$, then one can average $\eta$ on $G_0$, obtaining a $G_0$-invariant
primitive $\eta^{G_0}$ of $\mathrm{vol}_{\overline{\mathbf g}}$. Clearly, the volume functional defined by $\eta^{G_0}$ is also $G_0$-invariant.
\end{proof}
Non-compactness of the ambient manifold $\overline M$ and compactness of its isometry group is a rather restrictive
assumption;  we will now determine more general conditions that guarantee the existence of invariant volume functionals.
\begin{lem}\label{thm:charinvariantvolumes}
Let $U\subset\overline M$ be an open subset and let $\eta$ be any primitive of $\mathrm{vol}_{\overline{\mathbf g}}$ on $U$.
Consider the volume functional $\mathcal V(y)=\int_My^*(\eta)$, defined on the set of embeddings $\mathfrak M(U)$ of $M$ into
$\overline M$ having image in $U$, and let
\[\rho\colon\mathrm{Dom}(\rho)\subset\mathrm{Iso}(\overline M,\overline{\mathbf g})\times\mathfrak M(U)
\longrightarrow\mathfrak M(U)\]
be the natural local action by left-composition with isometries of $(\overline M,\overline{\mathbf g})$.

Then, for all $\phi\in\mathrm{Iso}(\overline M,\overline{\mathbf g})$, the map $y\mapsto\mathcal V(y)-\mathcal V(\phi\circ y)$ is
locally constant on $\mathrm{Dom}(\rho_\phi)$. If $\eta-\phi^*(\eta)$ is exact in $U$ and $x\in\mathfrak M(U)$, then
for $\phi$ sufficiently close to the identity, $\mathcal V(\phi\circ x)=\mathcal V(x)$.
\end{lem}
\begin{proof}
For all $\phi\in\mathrm{Iso}(\overline M,\overline{\mathbf g})$, $\phi^*(\eta)$ is a primitive of the volume form
$\mathrm{vol}_{\overline{\mathbf g}}$,
so that $\eta-\phi^*(\eta)$ is closed in its domain. If $x,y\in\mathfrak M(U)$ are $\mathcal C^0$-close,
then $x$ and $y$ are homotopic, hence, using Stokes' theorem, if $x$ and $y$ are in $\mathrm{Dom}(\rho_\phi)$,
$\int_Mx^*\big(\eta-\phi^*(\eta)\big)=\int_My^*\big(\eta-\phi^*(\eta)\big)$,
i.e., $\mathcal V(x)-\mathcal V(\phi\circ x)=\mathcal V(y)-\mathcal V(\phi\circ y)$.

If $\eta-\phi^*(\eta)$ is exact, then so is $x^*\big(\eta-\phi^*(\eta)\big)$, thus
$\int_Mx^*\big(\eta-\phi^*(\eta)\big)=0$, i.e., $\mathcal V(x)=\mathcal V(\phi\circ x)$.
Note that this equality holds also when $M$ has boundary, provided that $\phi\big(x(\partial M)\big)=x(\partial M)$.
\end{proof}

Lemma~\ref{thm:charinvariantvolumes} suggests how to construct invariant volume functionals.
The natural setup consists of a pair of open subset $U_1\subset U_2\subset\overline M$, an $m$-form $\eta$ on $U_2$ which is
a primitive of $\mathrm{vol}_{\overline{\mathbf g}}$ in $U_2$, with the following properties:
\begin{itemize}
\item $\phi(U_1)\subset U_2$ for $\phi$ in a neighborhood of the identity in $\mathrm{Iso}(\overline M,\overline{\mathbf g})$;
\item $\eta-\phi^*(\eta)$ is exact in $U_1$ for $\phi\in\mathrm{Iso}(\overline M,\overline{\mathbf g})$ near the identity.
\end{itemize}
\begin{cor}
Given objects $U_1$, $U_2$ and $\eta$ as above, the map $\mathcal V(x)=\int_Mx^*(\eta)$ is an invariant volume functional
in the set of embeddings $x\colon M\hookrightarrow\overline M$ having image contained in $U_1$.
\end{cor}
\begin{proof}
It follows immediately from the last statement of Lemma~\ref{thm:charinvariantvolumes}.
\end{proof}
Notice that $\eta-\phi^*(\eta)$ is closed, hence if $U_1$ has vanishing de Rham cohomology in dimension $m$, then it is exact.
This observation provides a large class of examples of manifolds $(\overline M,\overline{\mathbf g})$ where it is possible to define local
invariant volume functionals.
\begin{example}
If $\overline M$ is a non-compact manifold whose $m^{th}$ de Rham cohomology space is zero, then the volume functional defined by
any primitive of $\mathrm{vol}_{\overline{\mathbf g}}$ is invariant under the (whole) isometry group. More generally, if $x\colon M\hookrightarrow\overline M$ is an embedding
that has image contained in an open subset whose $m^{th}$ de Rham cohomology space is zero, then there exists a volume functional
invariant under isometries near the identity, defined in an open neighborhood $\mathcal U$ of $x$ in the set of embeddings of $M$
into $\overline M$. In particular, this applies when $\overline M$ is $\mathds R^{m+1}$ or $\overline M=S^{m+1}$.
Manifolds of the form $\overline M^{m+1}=\mathds R^k\times N^{m+1-k}$, $k\ge1$, have trivial $m^{th}$ de Rham cohomology space.
Manifolds of the form $\overline M^{m+1}=S^k\times N^{m+1-k}$, $k\ge1$, have open dense subsets with trivial $m^{th}$ de Rham cohomology space.
\end{example}
\end{section}

\end{document}